\documentclass[11pt]{article}
\usepackage{fullpage}
\usepackage{fancyhdr}
\usepackage{amsmath,amssymb,amsthm}
\usepackage{amsmath,amscd}
\usepackage{mathtools}
\usepackage{mathrsfs}
\usepackage{tikz-cd}
\usepackage{authblk}
\usepackage{enumitem}

\usepackage[letterpaper, headsep=1.5cm, margin=1.2in]{geometry}
\usepackage{baskervald}

\usepackage{hyperref}

\newcommand{\Z}{\mathbb{Z}}
\newcommand{\ord}{\text{ord}}

\newcommand{\R}{\mathbb{R}}
\newcommand{\Q}{\mathbb{Q}}

\newcommand{\gothm}{\mathfrak m}

\newcommand{\Sym}{\text{Sym}}

\newcommand{\Spec}{\text{Spec}}
\newcommand{\Proj}{\text{Proj}}
\newcommand{\Spf}{\text{Spf}}

\newcommand{\ang}[1]{\langle #1 \rangle}

\renewcommand {\bar}{\overline}

\newcommand{\Ext}{\text{Ext}}

\AtBeginDocument{}

\newtheorem{theorem}{Theorem}[section]
\newtheorem{lemma}[theorem]{Lemma}

\newtheorem{proposition}[theorem]{Proposition}

\newtheorem{corollary}[theorem]{Corollary}
\newtheorem{question}[theorem]{Question}

\theoremstyle{definition}

\newtheorem{example}[theorem]{Example}
\newtheorem{remark}[theorem]{Remark}
\newtheorem{definition}[theorem]{Definition}

\def\frak{\relaxnext@\ifmmode\let\next\frak@\else
	\def\next{\Err@{Use \string\frak\space only in math mode}}\fi\next}
\def\goth{\relaxnext@\ifmmode\let\next\frak@\else
	\def\next{\Err@{Use \string\goth\space only in math mode}}\fi\next}
\def\frak@#1{{\frak@@{#1}}}
\def\frak@@#1{\noaccents@\fam\euffam#1}
\font\tengoth=eufm10
\newfam\gothfam \def\goth{\fam\gothfam\tengoth} \textfont\gothfam=\tengoth

\title{$p$-adic estimates of exponential sums on curves}
\author{Joe Kramer-Miller}

\date{}

\begin{document}
	
	\maketitle
	
	\begin{abstract}
		The purpose of this article is to prove a ``Newton over Hodge'' result for 
		exponential sums on curves. Let $X$ be a smooth proper curve over a finite 
		field $\mathbb{F}_q$ of characteristic $p\geq 3$ and let $V \subset X$ be an 
		affine curve. For a regular function $\overline{f}$ on $V$, we may form the 
		$L$-function $L(\overline{f},V,s)$ associated to the exponential sums of 
		$\overline{f}$. In this article, we prove a lower estimate on the Newton 
		polygon of $L(\overline{f},V,s)$. The estimate depends on the local monodromy 
		of $f$ around each point $x \in X-V$. This confirms a hope of Deligne that 
		the irregular Hodge filtration forces bounds on $p$-adic valuations of 
		Frobenius eigenvalues. As a corollary, we obtain a lower estimate on the 
		Newton polygon of a curve with an action of $\Z/p\Z$ in terms of local 
		monodromy invariants.
	\end{abstract}
	
	\tableofcontents

	\section{Introduction}
	\subsection{Motivation}
	Let $p$ be a prime with $p\geq 3$ and let $q=p^a$. 
	Let $V$ be a $d$-dimensional smooth affine variety over $\mathbb{F}_q$
	and let $\overline{f}$ be a regular function on $V$. We define the 
	exponential sum over the $\mathbb{F}_{q^k}$-points of $V$
	to be 
	\begin{align*}
	S_k(\overline{f}) &= \sum_{x \in V(\mathbb{F}_{q^k}) } 
	\zeta_p^{Tr_{\mathbb{F}_{q^k}/\mathbb{F}_p}(\overline{f}(x))},
	\end{align*}
	where $\zeta_p$ is a primitive $p$-th root of unity.
	A fundamental question in number theory is to understand 
	the sequence of numbers $S_k(\overline{f})$ as $k$ varies. One approach is
	to study the generating $L$-function
	\begin{align*}
	L(\overline{f},V,s) &= \exp\Bigg ( \sum_{k=1}^\infty 
	\frac{S_k(\overline{f})s^k}{k} \Bigg).
	\end{align*}
	By the work of Dwork and Grothendieck, we have
	\begin{align*}
	L(\overline{f},V,s) &= \frac{\prod_{i=1}^{n_1}(1-\alpha_i 
		s)}{\prod_{i=1}^{n_2}(1-\beta_is)} 
	\in \Z[\zeta_p](s),
	\end{align*}  
	which implies
	\begin{align*}
	S_k(\overline{f}) &= \sum_{i=1}^{n_1} \alpha_i^{k} - \sum_{i=1}^{n_2} \beta_i^{k}.
	\end{align*}
	We are thus reduced to studying the algebraic integers $\alpha_i$ and $\beta_j$. 
	What can be said about $\alpha_i$ and $\beta_j$ in general? Deligne's work on the 
	Weil conjectures provides us with two stringent conditions.
	First, for $\ell\neq p$ these numbers are $\ell$-adic units:
	$|\alpha_i|_\ell = |\beta_j|_\ell=1.$
	Second, the $\alpha_i$ and $\beta_j$ are Weil numbers: there exists $u_i,v_j \in \Z 
	\cap [0,2d]$, known as weights, such that for any
	Archimedean absolute value $|\cdot |_\infty$ we have 
	$|\alpha_i|_\infty=q^{\frac{u_i}{2}}$ and $|\beta_j|_\infty= q^{\frac{v_j}{2}}$. 
	
	This leaves us with two natural questions. What are the 
	$p$-adic valuations of $\alpha_i$ and $\beta_j$, and what are the $u_i,v_j$? 
	If we take $V$ to be a multidimensional torus
	and make certain nonsingularity assumptions on $\overline{f}$, a
	great deal known about both questions. The $p$-adic valuations
	have been studied intensively by many
	(see, e.g., \cite{Adolphson-Sperber-exponential_sums} and 
	\cite{Wan-NP_zeta_functions} for two monumental works). The weights have been 
	computed by Adolphson-Sperber and Denef-Loeser (see 
	\cite{Adolphson-Sperber-exponential_sums} and 
	\cite{Denef-Loesler_weights_of_exponential_sums}).
	When $V$ is not a torus, it is possible to reduce to the toric
	case by an inclusion-exclusion argument. This allows us to
	write $L(\overline{f},V,s)$ as a ratio of $L$-functions of exponential sums 
	on tori. However, it is difficult to deduce precise results about 
	$L(\overline{f},V,s)$
	from this ratio, as cancellation often occurs.
	
	Now consider the case where $V$ is a smooth curve.
	In this situation $L(\overline{f},V,s)$ is a polynomial and 
	the weights are all one.
	Thus, we are left with the question of the $p$-adic valuations. 
	The purpose of this article is to prove lower
	bounds on the $q$-adic Newton polygon of $L(\overline{f},V,s)$. We prove a 
	``Newton over Hodge'' style result. This is in the vein of Mazur's celebrated
	theorem, which compares the Newton and Hodge polygons of a variety (see \cite{Mazur-NoverH}). 
	Our Hodge
	polygon is defined from the local monodromy at the poles of $g$. 
	Under a non-degeneracy assumption, this Hodge polygon is precisely the polygon 
	associated to the irregular
	Hodge filtration introduced by Deligne (see \cite[Th\'eorie de Hodge 
	Irr\'eguli\`ere]{Deligne-Malgrange-Ramis}).
	In particular, we confirm Deligne's hope that the irregular Hodge filtration
	forces bounds on $p$-adic valuations of Frobenius eigenvalues. 
	
	\subsection{Statement of main theorem}
	\label{subsection: statement of main theorem}
	\subsubsection{Exponential sums on curves}
	\label{subsubsection: exponential sums of curves}
	We now assume that $V=\Spec(B)$ is a smooth affine curve and we let $X$ be its 
	smooth compactification. Let $g$ be the genus of $X$. Assume that
	$\overline{f}$ is not of the form $x^p-x$ with $x \in B \otimes_{\mathbb{F}_q} 
	\overline{\mathbb{F}}_q$.
	We obtain a $\Z/p\Z$-cover of smooth proper curves $r:C \to X$ 
	from the equation
	\begin{align*}
	Y^p-Y&=\overline{f}, 
	\end{align*}
	and our condition on $\overline{f}$ implies that $C$ is geometrically connected.
	Let $\{\tau_1,\dots,\tau_{\mathbf{m}}\}\in X$ be the points
	over which $r$ ramifies and let $Z=X-\{\tau_1,\dots,\tau_{\mathbf{m}}\}$. 
	Note that $V$ is contained in $Z$.
	After increasing $q$, we may assume
	that each point in $X-V$ is defined over $\mathbb{F}_q$. 
	We remark that increasing $q$ does not alter the $q$-adic Newton polygons
	we are studying, so we will do so when convenient. 
	Consider a nontrivial character
	$\rho:Gal(C/X) \to \Z_p[\zeta_p]^\times$
	and its Artin $L$-function
	\begin{align} \label{introduction of L-function}
	L(\rho,s)&= \prod_{x \in Z} \frac{1}{1 - \rho(Frob_x) s^{\deg(x)}}.
	\end{align}
	Let $NP_q(L(\rho,s))$ (resp. $NP_q(L(\overline{f},V,s))$) denote the $q$-adic
	Newton polygon of $L(\rho,s)$ (resp. $L(\overline{f},V,s)$). 
	
	For some choice of $\zeta_p$, we have
	\begin{align*}
	S_k(\overline{f}) &=\sum_{x \in V(\mathbb{F}_{q^k}) } 
	\zeta_p^{Tr_{\mathbb{F}_{q^k}/\mathbb{F}_p}(\overline{f}(x))} = \sum_{x \in 
		V(\mathbb{F}_{q^k})} \rho(Frob_x),
	\end{align*}
	which gives the relation
	\begin{align*}
	L(\overline{f},V,s) &= L(\rho,s)\prod_{x \in Z-V} (1-\rho(Frob_x)s) .
	\end{align*}
	Therefore, we are reduced to studying $NP_q(L(\rho,s))$.
	
	Our main result gives a lower bound on $NP_q(L(\rho,s))$ in terms
	of the Swan conductor of $\rho$ at each point. 
	The Swan conductor $d_i$ at the point $\tau_i$ has a simple description. Let 
	$t_i$ be
	a local parameter at $\tau_i$. Locally 
	$r$ is given by an equation $Y^p-Y=g_i$, where $g_i \in \mathbb{F}_q((t_i))$.
	We may assume that $g_i$ has a pole whose order is prime to $p$.
	The order of this pole is equal to the Swan conductor. That is, 
	$g_i=\sum\limits_{n\geq-d_i} 
	a_nt_i^n$ and $a_{-d_i}\neq 0$. 
	
	\begin{theorem} \label{main theorem}
		The polygon $NP_q(L(\rho,s))$ lies above the polygon whose slopes are
		\begin{align}
		\Big\{\underbrace{0,\dots,0}_{g+{\mathbf{m}}-1}, \underbrace{1,\dots,1}_{g+{\mathbf{m}}-1} ,
		,\frac{1}{d_1}, \dots, \frac{d_1-1}{d_1}, \dots, 
		\frac{1}{d_{\mathbf{m}}}, \dots, \frac{d_{\mathbf{m}}-1}{d_{\mathbf{m}}}\Big\}. \label{hodge bound}
		\end{align}
		Let $c$ be the cardinality of $Z-V$. Then $NP_q(L(\overline{f},V,s))$
		lies above the polygon whose slopes are 
		\[ \Big\{\underbrace{0,\dots,0}_{g+{\mathbf{m}}+c-1}, \underbrace{1,\dots,1}_{g+{\mathbf{m}}-1} ,
		,\frac{1}{d_1}, \dots, \frac{d_1-1}{d_1}, \dots, 
		\frac{1}{d_{\mathbf{m}}}, \dots, \frac{d_{\mathbf{m}}-1}{d_{\mathbf{m}}}\Big\}.\]
	\end{theorem}
	
	\begin{remark}
		If $p\nmid \ord_{\tau_i}(\overline{f})$ for each $i$, the polygon with slopes 
		\eqref{hodge 
			bound} is precisely
		the polygon associated to the irregular Hodge filtration
		on twisted de Rham cohomology over a complex curve. One of Deligne's 
		motivations 
		for
		introducing the irregular Hodge filtration
		was a hope that it would force some bounds on the $p$-adic
		valuations of Frobenius eigenvalues. Theorem \ref{main theorem}
		affirms this prediction. We thank Jeng-Daw Yu for bringing
		this connection to our attention.
	\end{remark}
	
	\subsubsection{Zeta functions of Artin-Schreier covers}
	Theorem \ref{main theorem}
	also has interesting consequence about Newton polygons of 
	$\Z/p\Z$-covers of curves. Let $r: C \to X$ be a $\Z/p\Z$-cover
	ramified over $\tau_1,\dots,\tau_{\mathbf{m}}$ with Swan conductor
	$d_i$ at the point $\tau_i$. We let $g_C$ denote the genus of $C$.
	The zeta function of $C$ (resp. $X$) is a rational 
	function of the form $Z(C,s)=\frac{P_C(s)}{(1-s)(1-qs)}$ (resp.
	$Z(X,s)=\frac{P_X(s)}{(1-s)(1-qs)}$). 
	Let $NP_C$ (resp. $NP_X$) denote the $q$-adic Newton polygon
	of $P_C$ (resp. $P_X$).
	We are interested in the following question: to what extent can
	we determine $NP_C$ from $NP_X$ and the ramification invariants of $r$? 
	One positive result is the Deuring-Shafarevich
	formula (see \cite{Crew-p-covers}), which allows us
	to completely determine the number of slope zero segments in $NP_C$. However, a 
	precise
	formula for the higher slopes of $P_C$ is impossible. For example,
	work of Blach and F\`erard compute the generic Newton polygon for the moduli of maps
	$C \to \mathbb{P}^1$ ramified at only one point with a fixed ramification break $d$ (see \cite{Blache-Ferard-Newton_stratum}). They also show that if $p \not \equiv 1 \mod d$,
	there is a nontrivial locus where the Newton polygon jumps. This means
	that the ramification invariants are not enough to completely pin down
	the Newton polygon. Instead,
	the best we may hope for are bounds. Since
	the endpoints of $NP_C$ are $(0,0)$ and $(2g_C,g_C)$, we know
	that $NP_C$ lies below the polygon consisting of $2g_C$ segments of slope $\frac{1}{2}$.
	To find a lower bound, we may use Theorem \ref{main theorem}. Recall the
	decomposition:
	\begin{align*} 
	Z(C,s)&= Z(X,s) \prod_{\rho} L(\rho,s),
	\end{align*}
	where $\rho$ varies over the nontrivial characters 
	$Gal(C/X) \to \Z_p[\zeta_p]^\times$. This gives:
	\begin{corollary}
		The Newton polygon $NP_C$ lies above the polygon 
		whose slopes are the multiset: 
		\[ NP_X \bigsqcup \Bigg ( \bigsqcup_{i=1}^{p-1} 
		\Bigg\{\underbrace{0,\dots,0}_{g+{\mathbf{m}}-1}, 
		\underbrace{1,\dots,1}_{g+{\mathbf{m}}-1} ,
		,\frac{1}{d_1}, \dots, \frac{d_1-1}{d_1}, \dots, 
		\frac{1}{d_{\mathbf{m}}}, \dots, \frac{d_{\mathbf{m}}-1}{d_{\mathbf{m}}} \Bigg\}\Bigg),\]
		where $\sqcup$ denotes disjoint union.
	\end{corollary}

	\subsection{Idea of proof and previous work}
	\subsubsection{The Monsky trace formula}
	As our question is inherently $p$-adic in nature,
	it is necessary to utilize a $p$-adic Lefschetz trace formula.
	More specifically, we will utilize the Monsky trace formula
	(see \cite{Monsky-forma_cohomology3}), which generalizes
	the trace formulas of Dwork and Reich. Although the Monsky trace formula is often 
	viewed as a precursor to the trace formula of Etesse-Stum for rigid cohomology 
	(see \cite{Etesse-leStum}), 
	it is advantageous in that it allows us to compute on the level of chains. 
	The Monsky trace formula works roughly as follows (see \S \ref{subsection: MW 
		trace 
		formula} for a more precise formulation).
	For simplicity we will assume $q=p$. Let $\mathcal{X}^{rig}$ be a rigid analytic 
	lifting
	of $X$ defined over a finite extension $L$ of $\Q_p$ 
	and let $\mathcal{V}^{rig}$ be the the tube of $V \subset X$.
	Let $\mathcal{B}^\dagger$ denote the functions 
	on $\mathcal{V}^{rig}$ that overconverge in each tube $]\tau_i[$ and let $\sigma: 
	\mathcal{B}^\dagger \to \mathcal{B}^\dagger$ be a
	ring homomorphism that lifts the $p$-th power Frobenius map of $V$.
	Using $\sigma$ we define an operator $U_p: \mathcal{B}^\dagger\to 
	\mathcal{B}^\dagger$, which is the composition
	of a trace map $Tr:\mathcal{B}^\dagger\to \sigma(\mathcal{B}^\dagger)$ with 
	$\frac{1}{p}\sigma^{-1}$.

	The Galois representation $\rho$ corresponds to
	an overconvergent $F$-isocrystal with rank one. 
	This is a 
	$\mathcal{B}^\dagger$-module $M=\mathcal{B}^\dagger e_0$
	and a $\mathcal{B}^\dagger$-linear isomorphism $\varphi: M \otimes_{\sigma} 
	\mathcal{B}^\dagger \to M$. 
	Such an $F$-isocrystal is determined 
	entirely by some $\alpha \in \mathcal{B}^\dagger$ such that
	$\varphi(e_0 \otimes 1) = \alpha e_0$, which we refer to
	as the Frobenius structure of $M$.
	In our specific setup (see \S \ref{section: global bounds}),
	the Monsky trace formula can be written as follows:
	\begin{align} \label{MW:intro}
	L(\rho,s) &= \frac{\det(1-sU_p \circ \alpha |\mathcal{B}^\dagger)}
	{\det(1-spU_p \circ \alpha |\mathcal{B}^\dagger)},
	\end{align}
	where $\alpha$ means the ``multiplication by $\alpha$'' operator.
	To utilize \eqref{MW:intro}, we need to understand $U_p \circ \alpha$. This 
	breaks 
	up into two questions:
	\begin{question} \label{Question 1}
		Is there a Frobenius $\sigma$
		and a basis of $\mathcal{B}^\dagger$,
		for which the operators $U_p$ are
		reasonable to understand? 
	\end{question}
	\begin{question}\label{Question 2}
		Can we understand the Frobenius structure
		of $M$? In particular, we need to understand the
		``growth'' of the Frobenius structure in terms of a
		basis of $\mathcal{B}^\dagger$. 
	\end{question}
	
	\subsubsection{Previous work}
	Let us now recount what was previously known
	and earlier approaches to these questions. 
	When $X=\mathbb{P}^1_{\mathbb{F}_p}$ and $m=1,2$, Theorem
	\ref{main theorem} is due to Robba (see \cite{Robba-lower_bounds_NP}), building 
	off of ideas
	of Dwork. In this case, it is trivial
	to address Question \ref{Question 1}. Indeed, 
	we have $\mathcal{B}^\dagger = L\ang{t,t^{-1}}^\dagger$ and
	we may take $\sigma$ to be the map that sends $t \mapsto t^p$. 
	The Frobenius structure, which
	is known as a \emph{splitting function} in this case, may be described explicitly 
	using the Artin-Hasse exponential $E(t)$ (see \S \ref{subsection: frobenius 
		structures for asw extensions}). 
	The Taylor expansion of $E(t)$ lies in $\Z_p[[t]]$,
	which makes it particularly easy to estimate the Frobenius structure. 
	The work of Adolphson-Sperber and Wan for higher dimensional tori 
	also utilizes the Artin-Hasse exponential. However, the cohomological 
	calculations 
	used to find $p$-adic
	estimates are significantly more nuanced.
	
	The case where $X=\mathbb{P}^1_{\mathbb{F}_p}$
	and $m\geq 1$ was studied by Zhu 
	(see \cite{Zhu-exponential_sums_affinoids}).
	The key idea for addressing Question \ref{Question 1} is
	to decompose $\mathcal{B}^\dagger$ using
	partial fractions. The operator $U_p$
	can then be understood by modifying computations
	of Dwork (see \cite[Chapter 5]{Dwork-book}). To address Question \ref{Question 2},
	Zhu again utilizes the Artin-Hasse exponential, 
	analogous to Robba's splitting function. 
	However, when $m>2$ the Frobenius structure
	used in \cite{Zhu-exponential_sums_affinoids} does
	not correspond to a finite character, and thus
	calculates a different $L$-function. In particular Theorem \ref{main theorem} was 
	unknown for $m>2$. Nevertheless,
	Zhu's idea of computing the $L$-function by looking
	locally around each pole has influenced this article.
	
	\subsubsection{The approach to Questions \ref{Question 1} and \ref{Question 2} in 
		this article}
	The classical approach utilized by Adolphson-Sperber, Robba, and Wan, 
	no longer works when considering affine curves other than $\mathbb{G}_m$
	or $\mathbb{A}^1$. 
	First, there is no clear choice of Frobenius lift.
	Second, it is no longer clear how to make sense of the ring of functions
	$\mathcal{B}^\dagger$. Finally, the Frobenius structure of $M$ no longer has a 
	simple 
	global representation. This means that an entirely new method is needed to study
	the Fredholm determinant of $U_p \circ \alpha$.

	We first discuss our approach to Question \ref{Question 1}.
	Let $\overline{u}$ be a local parameter at a point $x \in X$ and let $u$ be a 
	local parameter that lifts $\overline{u}$.
	We need to find a Frobenius endomorphism 
	$\sigma:\mathcal{B}^\dagger \to \mathcal{B}^\dagger$,
	that behaves nicely with respect to $u$. 
	For a rational line with a global parameter
	$t$, we may take $\sigma$ to be the map that sends
	$t \mapsto t^p$. For a general curve $X$, we bootstrap
	from the rational case. To do this, we use a map $\eta: X \to 
	\mathbb{P}^1_{\mathbb{F}_p}$ satisfying certain properties: it is \`etale outside 
	of $\{0,1,\infty\}$,
	tamely ramified above $0$ and $\infty$, and every point in $\eta^{-1}(1)$ has 
	ramification index $p-1$. 
	Tdhe Frobenius endomorphism $t \mapsto t^p$ extends to $\sigma$
	on $\mathcal{B}^\dagger$. 
	If we choose our local parameters carefully, we only
	need to consider two types of local Frobenius endomorphisms: $u \mapsto u^p$ and $u 
	\mapsto
	\sqrt[p-1]{(u-1)^p+1}$. Furthermore, the map $\eta$ will be constructed so that 
	the local 
	Frobenius endomorphism at each $\tau_i$ is of the form $u \mapsto u^p$ (this 
	geometric setup is done in \S \ref{section: global bounds}). 
	In \S \ref{section: local Up operators} we estimate local versions of $U_p$
	for each type of local Frobenius endomorphism.
	This is one of the key technical obstacles for dealing with higher genus curves.
	
	Question \ref{Question 2} also requires a fresh approach. Previous work
	required a global \emph{splitting function} with nice properties. Over an 
	arbitrary curve $X$, there is no reason to expect a reasonable splitting function 
	exists. Instead, we have devised a method of computing the Fredholm determinant 
	of $U_p \circ \alpha$ using local Frobenius structures at each point (see \S 
	\ref{subsection: estimating NP_p}). These local Frobenius structures are fairly easy to 
	compute: if $x \in V$, then $\rho$ is unramified at $x$ and the local Frobenius 
	structure is an element of $\mathcal{O}_L^\times$. If $x \in 
	\{\tau_1,\dots,\tau_m\}$, then the local Frobenius endomorphism is of the form $u 
	\mapsto u^p$ by our construction of $\eta: X \to \mathbb{P}_{\mathbb{F}_p}^1$. We 
	can then use the Artin-Hasse function to construct the local Frobenius structure 
	(see \S \ref{subsection: frobenius structures for asw extensions}). In both 
	cases, 
	we have local Frobenius structures that are reasonable to compute with.
	
	\subsection{Further work}
	There are many natural questions that arise
	from Theorem \ref{main theorem}. The
	most pressing of which is to what extent Theorem \ref{main theorem}
	is optimal. Since $NP_X$
	has integer vertices
	we know that the vertices of $NP_q(L(\rho,s))$ have $y$-coordinates
	in $\frac{1}{p-1}\Z$. This implies that
	if $p \not\equiv 1 \mod d_i$ for some $i$,
	Theorem \ref{main theorem} can be slightly improved.
	For $V=\mathbb{A}^1$, Blach and F\`erard determine
	the generic Newton polynomial for a fixed $d$ (see 
	\cite{Blache-Ferard-Newton_stratum}). It would be interesting to consider higher 
	genus analogues of their result.
	When $V$ is $\mathbb{G}_m$, Robba proved (see 
	\cite{Robba-lower_bounds_NP})
	that the bound in Theorem \ref{main theorem} is
	attained if and only if $p \equiv 1 \mod d_i$
	for $i=1,2$. The obvious generalizations
	of this result
	using the bounds in Theorem \ref{main theorem} is false for general $X$. Indeed, 
	if
	$X$ is not ordinary, then by the Deuring-Shafarevich formula
	we know that the bound in Theorem \ref{main theorem}
	has too many slope zero segments. If $X$ is ordinary, do these generalizations 
	hold? Recently Jeremy Booher and 
	Rachel Pries have combined formal patching methods with Theorem \ref{main 
		theorem} have recently made progress in this direction, utilizing Theorem 
	\ref{main theorem} (see \cite{BOOHER2020}).
	When $X$ is not ordinary, can we replace
	some of the slope zero segments in Theorem \ref{main theorem}
	to obtain an optimal result?

	Another question would be to generalize 
	Theorem \ref{main theorem} to more general
	overconvergent $F$-isocrystals. Given an overconvergent
	$F$-isocrystal $M$ on $V$,
	can we bound $NP_q(L(\rho,s))$ in terms
	of local Swan conductors and perhaps the Frobenius slopes of $M$? 
	In light of recent work on crystalline companions by Abe (see
	\cite{Abe-crystalline_companions}), this would have profound
	consequences for the $L$-functions of $\ell$-adic sheaves and automorphic
	forms on function fields. 
	The methods developed in this article give
	a very general solution to Question \ref{Question 1}.
	Thus, to generalize Theorem \ref{main theorem}, the main difficulty lies in
	bounding the Frobenius structure of $M$. If
	$M$ is unit-root, we suspect that Theorem \ref{main theorem}
	has a direct generalization using the Swan conductors
	of the corresponding $p$-adic representation. 
	
	Finally, we mention our requirement that $p\geq 3$. When $p=2$
	it is likely that the methods in this paper still work.
	The main difficulty is that some estimates in \S
	\ref{section: local Up operators} must be modified. It is also not immediately 
	clear that we can find a cover $\eta: X \to \mathbb{P}_{\mathbb{F}_q}^1$ 
	satisfying the desired properties. To construct $\eta$, we use the fact that $X$ 
	admits a simply branched map to $\mathbb{P}_{\mathbb{F}_q}^1$, which is false 
	when $p=2$. However, upcoming work of
	of Kiran Kedlaya, Daniel Litt, and Jakub Witaszek provides a Belyi map in this 
	case. This should enough to handle the $p=2$ case.
	
	\subsection{Outline}
	We begin by describing our global setup in \S \ref{section: global bounds}.
	In \S \ref{section: local Up operators} we study the growth
	of local $U_p$ operators.
	In \S \ref{section: Local unit-root $F$-isocrystals}
	we introduce $F$-crystals, study the $F$-crystal associated to $\rho$,
	and study
	the growth of local $F$-crystals for Artin-Schreier covers.
	In \S \ref{section: newton polygons and functional analysis}
	we give some preliminary results on Fredholm determinants, Newton polygons, and Dwork operators. Finally, in \S \ref{section: proof of theorems} 
	we estimate the Fredholm determinant of $U_p \circ \alpha$ and 
	complete the proof of
	Theorem \ref{main theorem}.

	\subsection{Acknowledgments}
	Throughout the course of this work, we have
	benefited greatly from conversations with
	Daqing Wan, June Hui Zhu, James Upton,
	and Andrew Obus. This article has
	also benefited greatly from 
	comments and suggestions given by the anonymous referee.

	\section{Notation}
	
	\subsection{Conventions} \label{section: convensions}
	
	The following conventions will be used throughout the article.
	We let $\mathbb{F}_q$ be an extension of $\mathbb{F}_p$ with
	$a=[\mathbb{F}_q:\mathbb{F}_p]$. We remark that it suffices
	to prove Theorem \ref{main theorem}
	after replacing $q$ with a larger power of $p$. In particular, we will increase $q$ throughout the article 
	when it simplifies arguments. Let $L_0$ be the 
	unramified
	extension of $\Q_p$ whose residue field is $\mathbb{F}_q$. Let $E$
	be a totally ramified finite Galois extension of $\Q_p$ of degree $e$ and set 
	$L=E\otimes_{\Q_p} L_0$. We define $\mathcal{O}_L$ (resp. $\mathcal{O}_E$) to be the ring of integers of $L$ (resp. $E$) and let
	$\gothm$ be the maximal ideal of $\mathcal{O}_L$. We let $\pi_\circ$ be a uniformizing elemenet of $E$. Fix $\pi=(-p)^{\frac{1}{p-1}}$ and for any positive rational number 
	$s$ 
	we set 
	$\pi_s=\pi^{\frac{1}{s}}$.  We will assume that $E$ is large enough to contain 
	$\pi_{d_i}$ for each $i=1,\dots, \mathbf{m}$ as well as a $p$-th root of unity $\zeta_{p}$.
	Define $\nu$ to be
	the enodmorphism $\text{id}\otimes \text{Frob}$ of $L$,
	where $\text{Frob}$ is the $p$-Frobenius automorphism of $L_0$, and 
	let $\sigma=\nu^a$. 
	For any $E$-algebra $R$ and $x \in R$, we obtain an operator $R \to R$
	sending $r \mapsto xr$. By abuse of notation, we will refer to this operator as $x$. For any ring $R$ 
	with
	valuation $v:R \to \mathbb{R}$ and any $x \in R$ with
	$v(x)>0$, we let $v_x(\cdot)$ denote the normalization
	of $v$ satisfying $v_x(x)=1$.

	\subsection{Frobenius endomorphisms} \label{subsection: Frobenius endomorphisms}
	Let $\overline{A}$ be an $\mathbb{F}_q$-algebra,
	let $A$ be an $\mathcal{O}_L$-algebra with $A\otimes_{\mathcal{O}_L} \mathbb{F}_q = \overline{A}$,
	and let $\mathcal{A}=A \otimes_{\mathcal{O}_L} L$. 
	A $p$-Frobenius endomorphism (resp. $q$-Frobenius endomorphism) of $A$ is a ring endomorphism $\nu:A \to A$ (resp. $\sigma: A \to A$)
	that extends the map $\nu$ (resp. $\sigma$) on $\mathcal{O}_L$ defined in \S \ref{section: convensions} and reduces to the $p$-th power map (resp. $q$-th power map) of $\overline{A}$. 
	Note that $\nu$ (resp. $\sigma)$ extends to a map $\nu: \mathcal{A} \to \mathcal{A}$ (resp. $\sigma: \mathcal{A} \to \mathcal{A}$), which
	we also refer to as a $p$-Frobenius endomorphism (resp. $q$-Frobenius endomorphism) of $\mathcal{A}$.

	\subsection{Definitions of local rings} \label{subsection: basic definitions}
	We begin by defining some rings
	and modules, which will be used throughout this article.
	We define $L$-algebras:
	\[  \mathcal{E} = \Bigg\{ \sum_{-\infty}^\infty a_nt^n \Bigg |  
	\begin{array}  {l}
	\text{ We have } a_n\in L, ~\lim\limits_{n\to-\infty} v_p(a_n)=\infty,   \\
	\text{ and}	~	v_p(a_n) \text{ is bounded below.} 
	\end{array}
	\Bigg \},  
	\]\[
	\mathcal{E}^\dagger = \Bigg\{ \sum_{-\infty}^\infty a_nt^n  \in 
	\mathcal{E} \Bigg |  
	\begin{array}  {l}
	\text{ There exists $m>0$ such that} \\
	v_p(a_n) \geq -mn \text{ for $n\ll 0$} 
	\end{array}
	\Bigg \}.  
	\]
	We refer to $\mathcal{E}$
	as the Amice ring over $L$ with parameter $t$. Note that 
	$\mathcal{E}^\dagger$ and 
	$\mathcal{E}$  are local fields with residue field
	$\mathbb{F}_q((t))$. The valuation $v_p$ on $L$ extends to
	the Gauss valuation on each of these fields. 
	We define $\mathcal{O}_{\mathcal{E}}$ (resp. $\mathcal{O}_{\mathcal{E}^\dagger}$)
	to be the subring of $\mathcal{E}$ (resp. $\mathcal{E}^\dagger$) consisting
	of Laurent series with coefficients in $\mathcal{O}_L$. Note that
	if $\nu:\mathcal{E}\to\mathcal{E}$ is any $p$-Frobenius endomorphism,
	we have $\mathcal{E}^{\nu=1}=E$.
	For $m \in \Z$, we define the $L$-vector space of truncated
	Laurent series: 
	\[
	\mathcal{E}^{\leq m} = \Bigg\{ \sum_{-\infty}^\infty a_nt^n  \in 
	\mathcal{E} \Bigg |  
	\begin{array}  {l}
	a_n=0\text{ for all $n>m$} \\
	\end{array}
	\Bigg \}  .
	\]
	The space $\mathcal{E}^{\leq 0}$ is a ring and  $\mathcal{E}^{\leq m}$
	is an $\mathcal{E}^{\leq 0}$-module. There is a natural projection
	$\mathcal{E} \to \mathcal{E}^{\leq m},$
	given by truncating the Laurent series.

	\section{Global setup} \label{Section global bounds}
	\label{section: global bounds}
	We now introduce the global setup used to prove
	Theorem \ref{main theorem}. We adopt the notation from \S 
	\ref{subsection: statement of main theorem}.
	
	\subsection{Mapping to $\mathbb{P}^1$}
	\label{subsection: mapping to p1}
	\begin{lemma} \label{lemma: map to p1}
		After increasing $q$, there exists
		a tamely ramified morphism $\eta:X \to \mathbb{P}_{\mathbb{F}_q}^1$,
		ramified only above $0,1$, and $\infty$, such that $\tau_1,\dots, \tau_{\mathbf{m}} \in 
		\eta^{-1}(\{0,\infty\})$ and each $P \in \eta^{-1}(1)$ has ramification index 
		$p-1$.
	\end{lemma}
	
	\begin{proof}
		This is similar to \cite[Theor\'eme 5.6]{Saidi}. By \cite[Proposition 7.1]{Fulton-hurwitz},
		there exists a simply branched cover
		$f: X\times \Spec(\mathbb{F}_q^{alg}) \to \mathbb{P}^1_{\mathbb{F}_q^{alg}}$. 
		After increasing $q$, we may assume that $f$ descends to a map
		$f': X \to  \mathbb{P}^1_{\mathbb{F}_q}$. We also
		take $q$ large enough so that each $\tau_i$ and each branch point of $f$ is defined over $\mathbb{F}_q$.
		We may assume that $f'$
		is unramified over $0$ and $\infty$, and that
		$f(\tau_i)\neq 0,\infty$ for each $i$. This means that
		$f(\tau_i) \in \mathbb{G}_m(\mathbb{F}_q)$. After composing $f$
		with the $(q-1)$-th power map and a linear transformation, we obtain
		a map $g: X \to \mathbb{P}^1_{\mathbb{F}_q}$ that is only ramified
		at $1,2,$ and $\infty$. We then compose $g$ with the $(p-1)$-th power map to 
		obtain a map $h: X \to \mathbb{P}^1_{\mathbb{F}_q}$. Note that $h$ is only ramified 
		over $\{0,1,\infty\}$ and the ramification index of every point over $0$ is 
		$p-1$. Swap $0$ and $1$ with a linear transformation to obtain 
		$\eta$. 
	\end{proof}

	\subsection{Basic setup} \label{subsection: basic setup}
	Write
	$\mathbb{P}^1_{\mathbb{F}_q}=\Proj(\mathbb{F}_q[x_1,x_2])$
	and let $\bar{t}=\frac{x_1}{x_2}$ be a parameter at $0$. 
	Fix a morphism $\eta$ as in Lemma \ref{lemma: map to p1}. 
	For $* \in \{0,1,\infty\}$ we let
	$\{P_{*,1}, \dots, P_{*,r_*}\} = \eta^{-1}(*)$ and set $W= 
	\eta^{-1}(\{0,1,\infty\})$. Again, we will increase $q$ so that
	each $P_{*,i}$ is defined over $\mathbb{F}_q$. Fix $Q=P_{*,i} \in W$. We define $e_{Q}$ to be 
	the ramification
	index of $Q$ over $*$. From Lemma \ref{lemma: map to p1}, if $*=1$ we have 
	$e_{Q}=p-1$ 
	for 
	$1\leq i \leq r_1$, so that $r_1(p-1)=\deg(\eta)$. Also, by the Riemann-Hurwitz formula
	\begin{align}
	(g-1) + (r_0+r_1 + r_\infty) &= \deg(\eta)-g+1, \label{riemann-hurwitz eq}
	\end{align}
	where
	$g$ denotes the genus of $X$. Let $U=\mathbb{P}^1_{\mathbb{F}_q}-\{0,1,\infty \}$
	and $V=X-W$. Then $\eta: V \to U$ is a finite \`etale map of degree $\deg(\eta)$. 
	Let $\overline{B}$ (resp. $\overline{A}$) be the $\mathbb{F}_q$-algebra
	such that $V=\Spec(\overline{B})$ (resp. $U=\Spec(\overline{A})$).
	
	Let $\mathbb{P}^1_{\mathcal{O}_L}$ be the projective line
	over $\Spec(\mathcal{O}_L)$
	and let $\mathbf{P}^1_{\mathcal{O}_L}$ be the formal projective
	line over $\text{Spf}(\mathcal{O}_L)$. Let $t$ be a global parameter
	of $\mathbf{P}^1_{\mathcal{O}_L}$ lifting $\bar{t}$. By the deformation theory of 
	tame
	coverings (see \cite[Theorem 4.3.2]{Grothendieck-Murre-tame_fundamental_groups}) 
	there exists a tame cover
	$\mathbf{X} \to \mathbf{P}^1_{\mathcal{O}_L}$ 
	whose special fiber is $\eta$ and by formal GAGA (see 
	\cite[\href{https://stacks.math.columbia.edu/tag/09ZT}{Tag 
		09ZT}]{stacks-project}) 
	there exists a morphism
	of smooth curves $\mathbb{X} \to \mathbb{P}^1_{\mathcal{O}_L}$
	whose formal completion is $\mathbf{X} \to 
	\mathbf{P}^1_{\mathcal{O}_L}$.
	There exists local parameters $t_*$ and
	$u_{Q}$, which yield the diagram:
	
	\begin{equation*} 
	\begin{tikzcd}
	\widehat{\mathcal{O}}_{\mathbf{X},Q} \cong \mathcal{O}_L[[u_{Q}]] 
	\arrow[r] &\widehat{\mathcal{O}}_{X,Q} \cong 
	\mathbb{F}_q[[\bar{u}_{Q}]] \\
	\widehat{\mathcal{O}}_{\mathbf{P}^1_{\mathcal{O}_L},*} \cong 
	\mathcal{O}_L[[t_{*}]]
	\arrow[r]\arrow[u] & \widehat{\mathcal{O}}_{\mathbb{P}^1_{\mathbb{F}_q},*} \cong 
	\mathbb{F}_q [[\bar{t}_{*}]] \arrow[u].
	\end{tikzcd}
	\end{equation*}
	Our assumptions on the branching of $\eta$ allows 
	us to choose $u_{Q}$ such that $u_{Q}^{e_{Q}} = t_*$.
	We obtain an 
	$\mathcal{O}_L$-point
	of $\mathbb{P}^1_{\mathcal{O}_L}$ (resp. $\mathbb{X}$) by
	evaluating at $t_*=0$ (resp. $u_{Q}=0$):
	\begin{equation*} 
	\begin{tikzcd}
	\mathbb{X}  \arrow[d]& \Spec(\mathcal{O}_L[[u_{Q}]]) 
	\arrow[l] \arrow[d]& \Spec(\mathcal{O}_L)\arrow[l] \\
	\mathbb{P}^1_{\mathcal{O}_L} &	\Spec(\mathcal{O}_L[[t_*]]) \arrow[l]&
	\Spec(\mathcal{O}_L)\arrow[l]
	\end{tikzcd}
	\end{equation*}
	We denote the $\mathcal{O}_L$-point of $\mathbb{P}^1_{\mathcal{O}_L}$ (resp. 
	$\mathbb{X}$)
	by $[*]$ (resp. $[Q]$). After applying an automorphism of 
	$\mathbb{P}^1_{\mathcal{O}_L}$ we may assume 
	that $[0]=0$ and $[\infty]=\infty$ and $[1]=1$. Thus we
	may take $t_{0}=t$, $t_\infty = \frac{1}{t}$ and $t_1=t-1$. 
	
	Let
	$\mathbb{U} = \mathbb{P}^1_{\mathcal{O}_L} - \{ [0], [1],[\infty] \}$
	and $\mathbb{V} = \mathbb{X}-\{[R]\}_{R\in W}$, so that
	$\mathbb{\eta}: \mathbb{V} \to \mathbb{U}$ is \`etale. We define $\mathbf{U} = 
	\mathbf{P}^1_{\mathcal{O}_L} - \{ 0,1,\infty \}$
	and $\mathbf{V} = \mathbf{X}-\{R\}_{R\in W}$. Note that $\mathbf{U}$ (resp. 
	$\mathbf{V}$) is the formal completion of $\mathbb{U}$ (resp. $\mathbb{V}$). 
	Finally, 
	we
	let $\mathcal{U}^{rig}$ (resp. $\mathcal{V}^{rig}$) be the rigid analytic fiber 
	of $\mathbf{U}$  (resp. $\mathbf{V}$).

	\subsection{Local parameters and overconvergent rings}
	\label{subsection: expansion around local parameters}
	Let $A$ (resp. $\widehat{A}$ and $\widehat{\mathcal{A}}$) be the ring 
	of functions
	$\mathcal{O}_{\mathbb{U}}(\mathbb{U})$ (resp. 
	$\mathcal{O}_{\mathbf{U}}(\mathbf{U})$ 
	and $\mathcal{O}_{\mathcal{U}^{rig}}(\mathcal{U}^{rig})$) and let $B$
	(resp. $\widehat{B}$ and $\widehat{\mathcal{B}}$) be the ring of 
	functions
	$\mathcal{O}_{\mathbb{V}}(\mathbb{V})$ (resp. 
	$\mathcal{O}_{\mathbf{V}}(\mathbf{V})$ 
	and $\mathcal{O}_{\mathcal{V}^{rig}}(\mathcal{V}^{rig})$). 
	Let $\mathcal{E}_{*}$ (resp. $\mathcal{E}_{Q}$) be the Amice ring over $L$
	with parameter $t_{*}$ (resp. $u_{Q}$). By expanding functions 
	in terms of the $t_{*}$ and $u_{Q}$, we obtain the following diagrams:
	
	\begin{equation}\label{Local expansion commutative diagram}
	\begin{tikzcd}
	\widehat{B} \arrow[r]& 
	\bigoplus\limits_{R \in W} \mathcal{O}_{\mathcal{E}_{R}}  &
	\widehat{\mathcal{B}} \arrow[r] & \bigoplus\limits_{R \in W} 
	\mathcal{E}_{R} \\
	\widehat{A} \arrow[r]\arrow[u]& 
	\bigoplus\limits_{*\in\{0,1,\infty\}} \mathcal{O}_{\mathcal{E}_{*}} \arrow[u] 
	&
	\widehat{\mathcal{A}} \arrow[r]\arrow[u] & \bigoplus\limits_{*\in\{0,1,\infty\}} 
	\mathcal{E}_{*}\arrow[u].
	\end{tikzcd}
	\end{equation}
	We let $A^\dagger$ (resp. 
	$B^\dagger$) be the subring of
	$\widehat{A}$ (resp. $\widehat{B}$)
	consisting of functions that are overconvergent
	in the tube $]*[$ for each $*\in\{0,1,\infty\}$ (resp. $]R[$ for all $R \in W$).
	In particular, $B^\dagger$ fits into the following
	Cartesian diagram:
	\begin{equation} \label{overconvergent expansion diagram}
	\begin{tikzcd}
	B^\dagger \arrow[d]\arrow[r] & \bigoplus\limits_{R \in W} \arrow[d] 
	\mathcal{O}_{\mathcal{E}_{W}}^\dagger \\
	\widehat{B}\arrow[r] & \bigoplus\limits_{R \in W} \mathcal{O}_{\mathcal{E}_{W}} 
	.
	\end{tikzcd}
	\end{equation}
	Note that $A^\dagger$ (resp. $B^\dagger$) is the weak completion of $A$ (resp. 
	$B$) 
	in the sense of \cite[\S 2]{Monsky-Washnitzer-formal_cohomology1}. In particular,
	we have $A^\dagger =\mathcal{O}_L \Big < t,t^{-1}, \frac{1}{t-1} \Big>^\dagger$
	and $B^\dagger$ is an finite \'etale $A^\dagger$-algebra.
	Finally, we define $\mathcal{A}^\dagger$ (resp. $\mathcal{B}^\dagger$)
	to be $A^\dagger \otimes \Q_p$ (resp. $B^\dagger \otimes \Q_p$). 
	Then $\mathcal{A}^\dagger$ (resp. $\mathcal{B}^\dagger$) is
	equal to the functions in $\widehat{\mathcal{A}}$ (resp. $\widehat{\mathcal{B}}$) 
	that are overconvergent 
	in the tube $]*[$ for each $*\in\{0,1,\infty\}$ (resp. $]R[$ for all $R \in W$).

	\subsection{Global Frobenius and $U_p$ operators} \label{subsection Global Frobenius and Up}
	Let $\nu:\mathcal{A}^\dagger
	\to \mathcal{A}^\dagger$ be the $p$-Frobenius endomorphism that restricts to $\nu$ on $L$ 
	and sends $t$ to $t^p$.
	Let $\sigma=\nu^a$. For $* \in \{0,1,\infty\}$,
	we may extend
	$\nu$ to a $p$-Frobenius endomorphism of $\mathcal{E}_{*}^\dagger$,
	which we refer to as $\nu_{*}$. 
	In terms of the parameters $t_*$, these endomorphisms are given
	as follows:
	\[ t_0 \mapsto t_0^p,~~~t_\infty \mapsto t_\infty^p, \text{ and}~~~t_1\mapsto 
	(t_1-1)^p+1.\]
	
	Since the map $\widehat{A} \to
	\widehat{B}$ is \'etale and both rings are $p$-adically
	complete, we may extend both
	$\sigma$ and $\nu$ to maps 
	$\sigma,\nu:\widehat{B}\to\widehat{B}$. 
	This in turn gives $p$-Frobenius endomorphisms $\nu_{Q}$ of
	$\mathcal{E}_{Q}$, which make the diagrams in \eqref{Local expansion 
		commutative diagram} $p$-Frobenius equivariant. Furthermore, since 
	$\nu_{Q}$ 
	extends
	$\nu_{*}$, we know that $\nu_{Q}$ induces 
	a $p$-Frobenius endomorphism of $\mathcal{E}_{Q}^\dagger$. 	It follows from \eqref{overconvergent expansion diagram}
	that $\sigma$ and $\nu$ restrict to
	endomorphisms $\sigma, \nu: \mathcal{B}^\dagger \to \mathcal{B}^\dagger$.
	In terms
	of the parameters $u_{Q}$, the $p$-Frobenius endomorphisms 
	$\nu_{Q}$ can be described as follows:
	\begin{enumerate}
		\item When $*$ is $0$ or $\infty$, have $u_{Q}^{\nu_Q}=u_{Q}^p$, since
		$t_*^{\nu_*}=t_*^p$ and $u_{Q}^{e_Q}=t_*$. 
		\item When $*=1$, we have $u_{Q}^{\nu_Q} = \sqrt[p-1]{(u_Q^{p-1}-1)^p+1}$,
		since $t_1^{\nu_1}= (t_1-1)^p+1$ and $u_Q^{p-1}=t_1$. 
	\end{enumerate}
	Following \cite[\S 3]{van_der_Put-MW_cohomology}, there is a trace map
	$Tr_0: \mathcal{B}^\dagger \to \nu(\mathcal{B}^\dagger)$ (resp. $Tr: 
	\mathcal{B}^\dagger \to \sigma(\mathcal{B}^\dagger)$).
	We may define the $U_p$ operator on 
	$\mathcal{B}^\dagger$:
	\begin{align*}
	U_p: \mathcal{B}^\dagger &\to \mathcal{B}^\dagger \\
	x &\mapsto \frac{1}{p} \nu^{-1}(Tr_0(x)).
	\end{align*}
	Similarly, we define $U_q=\frac{1}{q}\sigma^{-1}\circ Tr$, so that
	$U_p^a=U_q$. Note that $U_p$ is $E$-linear and $U_q$ is $L$-linear.
	Both $U_p$ and $U_q$ extend to operators on $\mathcal{E}_Q^\dagger$,
	which makes the diagram \eqref{Local expansion commutative diagram} equivariant.

	\section{Local $U_p$ operators}
	\label{section: local Up operators}
	Let $\nu$ be any $p$-Frobenius endomorphism of $\mathcal{E}^\dagger$ (see \S \ref{subsection: Frobenius endomorphisms}). We define
	$U_p$ to be the map:
	\begin{align*}
	\frac{1}{p}\nu^{-1}\circ \text{Tr}_{\mathcal{E}^\dagger / 
		\nu(\mathcal{E}^\dagger)}: \mathcal{E}^\dagger \to \mathcal{E}^\dagger. 
	\end{align*}
	Note that $U_p$ is $\nu^{-1}$-semilinear (i.e. $U_p(y^\nu x)=yU_p(x)$ for 
	all $y\in 
	\mathcal{E}^\dagger$).
	In this section, we will study the growth of $U_p$ for the $p$-Frobenius
	endomorphisms of $\mathcal{E}^\dagger$ that appeared in \S \ref{subsection Global Frobenius and Up}. 
	
	\subsection{Some auxiliary rings}
	
	For any $k \in \Q$, we define the $k$-th partial valuation on 
	$\mathcal{E}$ as follows:
	for $x=\sum a_nt^n$ we set
	\[ w_k(x) = \min_{v_p(a_n) \leq k} \{n\}.\]
	These partial valuations satisfy:
	\begin{align*} 
	\begin{split}
	w_k(x+y) & \geq  \min(w_k(x),w_k(y)), \\
	w_k(xy) & \geq  \min_{i+j\leq k} (w_i(x) + w_j(y)).
	\end{split}
	\end{align*}
	Using these partial valuations we define
	\begin{align*}
	\mathcal{O}_{\mathcal{E}}^m(b) &= \Big\{ x  \in 
	\mathcal{O}_{\mathcal{E}} \Big |  
	\begin{array}  {l}
	w_k(x) \geq -km-b  \text{ for all $k\in\Q$}
	\end{array}
	\Big \},\\
	\mathcal{O}_{\mathcal{E}}^m &= \mathcal{O}_{\mathcal{E}}^m(0).
	\end{align*}
	Note that  $\mathcal{O}_{\mathcal{E}}^m$ is a ring and that
	\begin{align}\label{module multiplication eq}	
	x \in \mathcal{O}_{\mathcal{E}}^m(b_1), ~y\in \mathcal{O}_{\mathcal{E}}^m(b_2) 
	\implies xy \in \mathcal{O}_{\mathcal{E}}^m(b_1 + b_2).
	\end{align}
	An alternative definition of $\mathcal{O}_{\mathcal{E}}^m(b)$ is 
	\[
	\mathcal{O}_{\mathcal{E}}^m(b) = \Bigg\{ \sum_{-\infty}^\infty a_nt^n  
	\in 
	\mathcal{O}_{\mathcal{E}^\dagger} \Bigg |  
	\begin{array}  {l}
	\text{ For all $n<0$ we have} \\
	v_p(a_{n-b}) \geq -\frac{1}{m}n 
	\end{array}
	\Bigg \}.  
	\]

	\begin{lemma} \label{lemma: mult or divide by p growth}
		We have:
		\begin{enumerate}
			\item	If $x\in \mathcal{O}_{\mathcal{E}}^m(b)$,
			then $px \in \mathcal{O}_{\mathcal{E}}^m(b-m)$.
			\item If $x \in \mathcal{O}_{\mathcal{E}}^m(b)$
			and $p|x$, then $\frac{x}{p} \in \mathcal{O}_{\mathcal{E}}^m(b+m)$.
			\item If $t^\nu \in \mathcal{O}_{\mathcal{E}}^m(-p)$, then
			$\nu(\mathcal{O}_{\mathcal{E}}^{\frac{m}{p}}) \subset \mathcal{O}_{\mathcal{E}}^m$.
		\end{enumerate}
	\end{lemma}
	\noindent We will use the next lemma in \S \ref{subsection: type 2 frobenius}
	and \S \ref{subsubsection: Local and semi-local Dwork operators}.
	\begin{lemma} \label{lemma: decomposition of frobenius action}
		Let $\nu$ be a $p$-Frobenius endomorphism of $\mathcal{E}^\dagger$
		and let $m\geq 0$ with $t^{\nu} \in \mathcal{O}_{\mathcal{E}}^m(-p)$.
		Then we have
		\begin{align*}
			\mathcal{O}_{\mathcal{E}}^m &=\bigoplus_{i=0}^{p-1} t^i \nu(\mathcal{O}_{\mathcal{E}}^{\frac{m}{p}}), \\
			\mathcal{O}_{\mathcal{E}}^m(p) &= \nu(\mathcal{O}_{\mathcal{E}}^{\frac{m}{p}}(1))
			\oplus \bigoplus_{i=1}^{p-1} t^{-i} \nu(\mathcal{O}_{\mathcal{E}}^{\frac{m}{p}}).
		\end{align*}
	\end{lemma}
	\begin{proof}
		We prove the first decomposition only, as the second one is almost identical. 
		Recall that $\pi_\circ$ is a uniformizing element of $\mathcal{O}_L$ and
		$v_p(\pi_\circ)=\frac{1}{e}$. Let $x \in \mathcal{O}_{\mathcal{E}}^m$.
		We will prove inductively that there exists $y_{n,i} \in \mathcal{O}_{\mathcal{E}}^{\frac{m}{p}}$ with
		\begin{align} \label{what we are proving}
			\sum_{i=0}^{p-1} t^iy_{n,i}^\nu \equiv x \mod \pi_\circ^{n+1} \mathcal{O}_{\mathcal{E}^\dagger}.
		\end{align}
		For $n=0$, write
		\begin{align*}
			x \equiv \sum_{i=0}^{p-1} t^i\overline{y}_{0,i}^p \mod \pi_\circ \mathcal{O}_{\mathcal{E}^\dagger},
		\end{align*}
		where $\overline{y}_{0,i}\in \mathbb{F}_q[[t]]$. We take $y_{0,i}$
		to be a lift of $\overline{y}_{0,i}$ contained in $\mathcal{O}_{\mathcal{E}}^{\frac{m}{p}}$.
		Now let $n>0$ and assume that there exists $y_{n,i}\in \mathcal{O}_{\mathcal{E}}^{\frac{m}{p}}$ that satisfy \eqref{what we are proving}.
		By Lemma \ref{lemma: mult or divide by p growth} and \eqref{module multiplication eq} we see that $\sum_{i=0}^{p-1} t^i y_{n,i}^\nu\in
		\mathcal{O}_{\mathcal{E}}^m$. We have
		\begin{align}\label{Frobenius decomposiiton eq1}
			x-\sum_{i=0}^{p-1} t^i y_{n,i}^\nu \equiv \pi_\circ^{n+1}\overline{r}_n, 
			\mod \pi_\circ^{n+1} \mathcal{O}_{\mathcal{E}^\dagger}.
		\end{align}
		where $\overline{r}_n \in \mathbb{F}_q((t))$. The left side of \eqref{Frobenius decomposiiton eq1} lies in 
		$\mathcal{O}_{\mathcal{E}}^m$. Thus, we have 
		$v_t(\overline{r}_n)\geq \frac{(n+1)m}{e}$. As in the $n=0$ case,
		we write
		\begin{align*}
			\overline{r}_n&=\sum_{i=0}^{p-1}t^i\overline{r}_{n,i}^p,
		\end{align*}
		and note that $v_t(\overline{r}_{n,i}) \geq \frac{m(n+1)}{ep}$. Then
		by Lemma \ref{lemma: mult or divide by p growth}, there
		exists a lift $r_{n,i}$ of $\overline{r}_{n,i}$
		such that $\pi_\circ^{n+1}r_{n,i}  $ is contained $\mathcal{O}_{\mathcal{E}}^{\frac{m}{p}}$.
		We take $y_{n+1,i}=y_{n,i}+\pi_\circ^{n+1}r_{n,i}$.
	\end{proof}

	\subsection{Type 1: $t \mapsto t^p$}
	
	First consider the $p$-Frobenius endomorphism $\nu:\mathcal{E}^\dagger \to \mathcal{E}^\dagger$ 
	sending $t$ to $t^p$. 
	In this case, we compute:
	\begin{align*}
	U_p(t^i) &= \begin{cases}
	0 & p \nmid i \\
	t^{\frac{i}{p}} & p | i 
	\end{cases}.
	\end{align*}
	Thus, for $m\geq 0$ we have:
	\begin{align}
	U_p(\mathcal{O}_\mathcal{E}^m) \subset \mathcal{O}_\mathcal{E}^{\frac{m}{p}}. 
	\label{basic p-frob 
		growth}
	\end{align}
	\begin{proposition} \label{type 1 operator property}
		Let $\nu$ be the $p$-Frobenius endomorphism of $\mathcal{E}^\dagger$
		that sends $t$ to $t^p$. Let $s$ be a positive rational number.
		Recall that $\pi_s=\pi^{\frac{1}{s}}$, so that $v_p(\pi_s)=\frac{1}{s(p-1)}$,
		and we assume $\pi_s \in \mathcal{O}_L$. Let $x \in \mathcal{O}_{\mathcal{E}}^{s(p-1)}$. Then
		\begin{align*}
		U_p(\pi_s^{pn}t^{-n} x) \in 
		\pi_s^{(p-1)n} \mathcal{O}_{\mathcal{E}}^{\frac{s(p-1)}{p}}.
		\end{align*}
	\end{proposition}
	\begin{proof}
		We have $U_p(\pi_s^{pn}t^{-n}x)=\pi_s^{(p-1)n}U_p(\pi_s^nt^{-n}x)$. As 
		$\pi_s^nt^{-n}x \in  \mathcal{O}_{\mathcal{E}}^{s(p-1)}$, the result follows 
		from \eqref{basic p-frob growth}.
	\end{proof}
	\subsection{Type 2: $t \mapsto \sqrt[p-1]{(t^{p-1}+1)^p-1}$}
	\label{subsection: type 2 frobenius}
	Next, consider the $p$-Frobenius endomorphism $\nu:\mathcal{E}^\dagger \to \mathcal{E}^\dagger$ that
	sends
	$t$ to $\sqrt[p-1]{(t^{p-1}+1)^p-1}$.
	Let $u=t^{p-1}$ and let $\mathcal{E}_0$ be the Amice ring with parameter $u$. Then
	$\mathcal{E}$ is a Galois extension of $\mathcal{E}_0$ with 
	$Gal(\mathcal{E}/\mathcal{E}_0) \cong \Z/(p-1)\Z$. Each summand of
	the decomposition 
	\begin{align}
	\mathcal{E} &= \bigoplus_{i=1}^{p-1} t^{-i} \mathcal{E}_0 \label{type 2 
		decomposition}
	\end{align}
	is a $\chi$-eigenspace for some $\chi:Gal(\mathcal{E}/\mathcal{E}_0) \to 
	\mu_{p-1}$. 
	Note that $\nu$ commutes with $Gal(\mathcal{E}/\mathcal{E}_0)$. Thus, the 
	decomposition \eqref{type 2 decomposition} is preserved by $\nu$ and 
	$U_p$. Define
	$\mathcal{B}=\mathcal{E}_0 \cap \mathcal{O}_\mathcal{E}^{p-1}$, so that
	\begin{align*}
	\mathcal{B} = \Bigg \{ \sum_{-\infty}^\infty a_n u^n \Bigg | a_n \in 
	\mathcal{O}_L \text{ and } v_p(a_n)\geq \max\{0,-n\} \Bigg \}.
	\end{align*}
	We then define the space
	\begin{align}
	\mathcal{A} &= \bigoplus_{i=1}^{p-1} t^{-i} \mathcal{B}. \label{definition of A}
	\end{align}
	We may describe an $\mathcal{O}_L$-basis of $\mathcal{A}$ as follows: 
	consider the 
	sequence
	\begin{align*}
	a(n) = \Big \lfloor \frac{n-1}{p-1} \Big \rfloor \text{ for }n\geq 1.
	\end{align*}
	Then $\{\dots, t^2, t^1, 1, 
	p^{a(1)}t^{-1},p^{a(2)}t^{-2},\dots\}$ 
	is an $\mathcal{O}_L$-basis of $\mathcal{A}$. 
	\begin{proposition} \label{proposition: type 2 Frobenius}
		Let $\nu$ be the $p$-Frobenius endomorphism of $\mathcal{E}^\dagger$ that sends
		$t$ to $\sqrt[p-1]{(t^{p-1}+1)^p-1}$. For all $n\in \Z_{\geq 0}$ and $0\leq j \leq p-1$, we have
		\begin{align*}
		U_p(p^{a(j+np)}t^{-(j+np)}) &\in p^n \mathcal{A}, \\
		U_p(\mathcal{A}) &\subset \mathcal{A}. 
		\end{align*}
	\end{proposition}
	
	\noindent The remainder of this section is dedicated to proving Proposition 
	\ref{proposition: 
		type 2 Frobenius}.
	
	\begin{lemma} \label{lemma: growth of frobenius}
		We have $t^\nu \in 
		\mathcal{O}_{\mathcal{E}}^{p(p-1)}(-p)$. 
	\end{lemma}
	\begin{proof}
		Note that $(t^{p-1}+1)^p-1\in 
		\mathcal{O}_{\mathcal{E}}^{p(p-1)}(-p(p-1))$. This means 
		$t^{-(p-1)p}((t^{p-1}+1)^p-1)=1+py$, where
		$py \in \mathcal{O}_{\mathcal{E}}^{p(p-1)}$. This gives
		\begin{align*}
		t^{-p}t^{\nu} &= \sqrt[p-1]{1+py} \\
		&= \sum_{n=0}^\infty \binom{\frac{1}{{p-1}}}{n}p^ny^n.
		\end{align*}
		This power series converges in $\mathcal{O}_{\mathcal{E}}^{p(p-1)}$, 
		since the 
		binomials
		$\binom{\frac{1}{{p-1}}}{n}$ have nonnegative $p$-adic valuation.
		Therefore $t^{\nu}\in 
		\mathcal{O}_{\mathcal{E}}^{p(p-1)}(-p)$. 
	\end{proof}
	
	%
	%
	%

	\begin{lemma} \label{lemma: growth of coefficients of min poly of t}
		Let $Q(X)=X^p + a_1^\nu X^{p-1} + \dots +a_p^\nu$ be the minimal 
		polynomial of 
		$t^{-1}$ over 
		$\nu(\mathcal{E})$. Then $a_p \in t^{-1} \mathcal{B}$ and 
		$\frac{a_i}{p} \in 
		t^{-i} \mathcal{B}$ for $1\leq i \leq p-1$. 
	\end{lemma}
	
	\begin{proof}
		From Lemma \ref{lemma: growth of frobenius} and Lemma \ref{lemma: decomposition of frobenius action}
		we know that $-t^p=a_1^\nu t^{p-1} + \dots +a_p^\nu$, where $a_p \in \mathcal{O}_{\mathcal{E}}^{p-1}(1)$ and $a_i \in \mathcal{O}_{\mathcal{E}}^{p-1}$.
		In particular, $Q(X)=X^p + a_1^\nu X^{p-1} + \dots +a_p^\nu$ is the minimal 
		polynomial of 
		$t^{-1}$ over 
		$\nu(\mathcal{E})$. Furthermore, we have $Q(X) \equiv X^p - t^p \mod p$ since
		$t^{\nu} \equiv t^p \mod p$,
		which implies $p|a_i$ for $i=1,\dots,p-1$. Then
		from Lemma \ref{lemma: mult or divide by p growth}, 
		we see that $\frac{a_i}{p} \in 
		\mathcal{O}_{\mathcal{E}}^{p-1}(p-1)$ for $i=1,\dots,p-1$. To finish the proof, let 
		$\beta_1, \dots, \beta_p$ be the Galois 
		conjugates of $t^{-1}$ over $\nu(\mathcal{E})$. Since $\nu$ 
		commutes with $Gal(\mathcal{E}/\mathcal{E}_0)$, we know that
		each $\beta_i$ is in the same summand of \eqref{type 2 decomposition}.
		That is, $\beta_i \in t^{-1}\mathcal{E}_0$. Since $a_i^\nu$
		is a symmetric function of the $\beta_i$ of degree $i$, we see that
		$\frac{a_i}{p} \in t^{-i}\mathcal{E}_0$ for $1\leq i \leq p-1$ and $a_p 
		\in 
		t^{-1}\mathcal{E}_0$. The result follows by observing that
		$t^{-i}\mathcal{E}_0\cap \mathcal{O}_{\mathcal{E}}^{p-1}(p-1) = 
		t^{-i}\mathcal{B}$ for $i=1,\dots, p-1$. 
		
	\end{proof}
	
	\begin{corollary} \label{corollary: basic growth of Up}
		For all $n\in \Z$ and $0\leq j \leq p-1$ we have $U_p(t^{-(j+np)}) \in 
		t^{-(j+n)}\mathcal{B}$. 
	\end{corollary}
	
	\begin{proof}
		We will only prove the case where $n\geq 0$. The proof for $n<0$ is a similar 
		symmetric polynomial argument. Let $\beta_1,\dots,\beta_p$ be the Galois 
		conjugates of $t^{-1}$ over $\nu(\mathcal{E})$, so that 
		$U_p(t^{-k})^\nu=\frac{\sum \beta_i^k}{p}$. 
		Then the $a_i$ from Lemma \ref{lemma: growth of coefficients of min poly 
			of t} are symmetric in $\beta_1,\dots,\beta_p$:
		\begin{align*}
		a_k^\nu&= \sum_{i_1<\dots<i_k} \beta_{i_1}\dots\beta_{i_n} \text{ for 
			$k=1,\dots, 
			p$}.
		\end{align*}
		The Newton identities give 
		\begin{align*}
		U_p(t^{-k})&=(-1)^{k-1}k\frac{a_k}{p} +  
		\sum_{i=1}^{k-1}(-1)^{k-1+i}a_{k-i}U_p(t^{-i}) ~~\text{   
			for $k\leq p$}\\
		U_p(t^{-k})&= 
		\sum_{i=k-p}^{k-1}(-1)^{k-1+i}a_{k-i}U_p(t^{-i}) 
		~~\text{   for $k>p$}.    
		\end{align*}
		The result follows from Lemma \ref{lemma: growth of coefficients of min 
			poly of t} by inducting on $k$. 
	\end{proof}
	
	\begin{proof} (Of Proposition \ref{proposition: type 2 Frobenius})
		First note that $a(j+np)-a(j+n)=n$. Then we have
		\begin{align*}
		U_p(p^{a(j+np)}t^{-(j+np)}) &\in p^n p^{a(j+n)} t^{-(j+n)} \mathcal{B} 
		\subset p^n \mathcal{A}. 
		\end{align*}
		The second part of the proposition is immediate from Corollary 
		\ref{corollary: basic growth of Up}.
	\end{proof}

	\section{Unit-root $F$-crystals}
	\label{section: Local unit-root $F$-isocrystals}	
	
	\subsection{Rank one $F$-crystals and $p$-adic characters}
	\label{subsection: local rank one crystals}
	Let $\overline{S}$ be either $\Spec(\mathbb{F}_q((t)))$ or a smooth, reduced,
	irreducible affine $\mathbb{F}_q$-scheme $\Spec(\overline{R})$. We let $S=\Spec(R)$
	be a flat $\mathcal{O}_L$-scheme whose special fiber is $\overline{S}$ and assume that $R$ is $p$-adically complete (if $\overline{S}=\Spec(\mathbb{F}_q((t)))$ then we may take
	$R=\mathcal{O}_\mathcal{E}$).
	Fix a $p$-Frobenius endomorphism $\nu$ on $R$ (as in \S \ref{subsection: Frobenius endomorphisms}). Then $\sigma=\nu^a$
	is a $q$-Frobenius endomorphism. 
	\begin{definition}
		A $\varphi$-module for $\sigma$ over $R$ is a locally
		free $R$-module $M$
		equipped with a $\sigma$-semilinear endomorphism $\varphi: M \to M$. 
		More precisely,
		we have $\varphi(cm)=\sigma(c)\varphi(m)$ for $c\in R$.  
	\end{definition}	
	\begin{definition} \label{varphi module definition}
		A unit-root $F$-crystal $M$ over $\overline{S}$ is
		a $\varphi$-module over $R$ such that $\sigma^* \varphi: R \otimes_{\sigma} M \to M$
		is an isomorphism. The rank of $M$ is defined to be the rank of the underlying
		module $M$. 
	\end{definition}
	\noindent The following theorem is a characteristic $p$ version
	of the Riemann-Hilbert correspondence.
	\begin{theorem}[Katz, see \S 4 in \cite{Katz-p-adic_properties}] \label{mod p 
			riemann-hurwitz}
		There is an equivalence of categories
		\begin{align*}
		\{\text{rank one unit-root $F$-crystals over $\overline{S}$}\} 
		&\longleftrightarrow \{\text{continuous characters} ~ \psi:\pi_1(\overline{S}) \to \mathcal{O}_L^\times\}, 
		\end{align*}
		where $\pi_1(\overline{S})$ is the \`etale fundamental group of $\overline{S}$. 
	\end{theorem}	

	 Let us describe this correspondence in a certain case. Let
	 $\overline{S}_1 \to \overline{S}$ be a finite \`etale cover and assume that
	 $\psi$ comes from a map $\psi_0:Gal(\overline{S}_1/\overline{S})\to \mathcal{O}_{E}^\times$.
	 This cover deforms into a finite \`etale map of affine formal schemes
	$S_1=\Spf(R_1) \to S$. Both $\nu$ and $\sigma$ extend to $R_1$ and commute with
	the action of $Gal(\overline{S}_1/\overline{S})$ (see e.g. \cite[\S 2.6]{Tsuzuki-finite_local_monodromy}). Let 
	$V_0=\mathcal{O}_{E}e_0$ be a rank one $\mathcal{O}_E$-module, on which $Gal(\overline{S}_1/\overline{S})$ acts on via $\psi_0$, and let $V=V_0 \otimes_{\mathcal{O}_E}\mathcal{O}_L$.  The unit-root $F$-crystal associated to is $\psi$ is $M_\psi=(R_1 \otimes_{\mathcal{O}_L} V)^{Gal(\overline{S}_1/\overline{S})}$ with $\varphi=\sigma\otimes_{\mathcal{O}_L} id$. 
	There is a Galois equivariant isomorphism 
	\[ (S_1 \otimes_{\mathcal{O}_{E}} V_0) \to (S_1 
	\otimes_{\mathcal{O}_L} V).\]
	In particular, the map $\varphi$ has an $a$-th root $\varphi_0=\nu \otimes_{\mathcal{O}_{E}} id$.
	
	Now make the additional assumption that $M_\psi$ is free as an $R$-module.
	Write $M_\psi=Re_0$. Then we have $\varphi(e_0)=\alpha e_0$ (resp. $\varphi_0(e_0)=\alpha_0e_0$), where $\alpha \in R^\times$. 
	We refer to $\alpha$ (resp. $\alpha_0$) as a \emph{Frobenius structure} 
	(resp. $p$\emph{-Frobenius structure}) of $M$. We have the relation $\alpha=\prod_{i=0}^{a-1} \alpha_0^{\nu^i}$.  
	If $e_1=be_0$ with $b \in R^\times$, and
	$\varphi(e_1)=\alpha'e_1$ (resp. $\varphi(e_1)=\alpha_0'e_1$), then we have $\alpha'=\frac{b^\sigma}{b}\alpha$ (resp. $\alpha_0'=\frac{b^\nu}{b}\alpha_0$).
	In particular, a Frobenius structure (resp. $p$-Frobenius structure)
	of $M$ is unique up to multiplication by
	an element of the form $\frac{b^\sigma}{b}$ (resp. $\frac{b^\nu}{b}$) with $b \in R^\times$.

	\subsection{Some local Frobenius structures}
	We now restrict ourselves to the case where $\overline{S}=\Spec(\mathbb{F}_q((t)))$,
	so that rank one unit-root $F$-crystals over $\overline{S}$ correspond to characters of
	$G_{\mathbb{F}_q((t))}$, the absolute Galois group of $\mathbb{F}_q((t))$.
	We let $R=\mathcal{O}_{\mathcal{E}}$. Since $\mathcal{O}_{\mathcal{E}}$ is a local ring, all locally
	free modules of rank one are isomorphic to $\mathcal{O}_{\mathcal{E}}$.
	
	\subsubsection{Unramified Artin-Schreier characters}
	\begin{proposition} \label{proposition: unramified character proposition}
		Let $\nu$ be any $p$-Frobenius endomorphism of $\mathcal{O}_\mathcal{E}$
		and let $\sigma=\nu^a$. 
		Let $\psi:G_{\mathbb{F}_q((t))} \to \mathcal{O}_L^\times$ and let $M_\psi$ be the
		corresponding unit-root $F$-crystal. Assume that $Im(\psi)\cong \Z/p\Z$
		and that $\psi$ is unramified. There exists a $p$-Frobenius structure
		$\alpha_0$ of $M_\psi$ with $\alpha_0 \in 1+\gothm$ (recall $\gothm$ is the
		maximal ideal of $\mathcal{O}_L$). Furthermore, if $c \in 1+\gothm\mathcal{O}_\mathcal{E}$ is another $p$-Frobenius structure
		of $M_\psi$, there exists $b\in 1+\gothm\mathcal{O}_\mathcal{E}$ with $\alpha_0=\frac{b^\nu}{b}c$.
	\end{proposition}

	\begin{proof}
		Since $\psi$ is unramified it factors through a representation $\psi_1:\pi_1(\Spec(\mathbb{F}_q)) \to \mathcal{O}_L^\times$.
		Note that $M_{\psi_1}$ is a $\varphi$-module over $\mathcal{O}_L$
		and we have $M_{\psi_1} \otimes_{\mathcal{O}_L} \mathcal{O}_\mathcal{E}=M_\psi$. 
		Thus, $M_\psi$ has a $p$-Frobenius structure $\alpha_0$ contained in $\mathcal{O}_L^\times$.
		Since $\psi$ reduces to the trivial character modulo $\gothm$, we may take
		$\alpha_0$ to be in $1 +\gothm$. Now, let $c\in 1+\gothm\mathcal{O}_\mathcal{E}$
		be another $p$-Frobenius structure. There exists $b\in \mathcal{O}_\mathcal{E}^\times$
		with $\alpha_0=\frac{b^\nu}{b}c$, and $b$ is unique up to multiplication by an element in $\mathcal{O}_E^\times$ (recall from \S \ref{subsection: basic definitions} 
		that $E=\mathcal{E}^{\nu=1}$). We have 
		\begin{align*}
			b^{\nu}\equiv b \mod \gothm\mathcal{O}_\mathcal{E},
		\end{align*}
		which implies $b \equiv x \mod  \gothm\mathcal{O}_\mathcal{E}$ for some $x \in \mathbb{F}_p$.
		It follows that $[x^{-1}]b \in 1+\gothm\mathcal{O}_\mathcal{E}$,
		where $[x^{-1}] \in \mathcal{O}_E$ is the Teichmuller lift of $x$. 
	\end{proof}

	\subsection{Wild Artin-Schreier characters}
	\label{subsection: frobenius structures for asw extensions}
	The following result is commonplace in the literature, but not exactly presented 
	in this form (see e.g. \cite[\S 4.1]{Wan-variationNP} for an overview).
	We let $E(x)$ denote the Artin-Hasse exponential and let $\gamma$
	be an element of $\Z_p[\zeta_p]$ with $E(\gamma)=\zeta_{p}$. Note that
	$v_p(\gamma)=\frac{1}{(p-1)}$. 
	
	\begin{proposition} \label{theorem: ASW frobenius structure}
		Let $\nu$ be the $p$-Frobenius endomorphism sending $t$ to $t^p$ and
		let $\sigma=\nu^{a}$. Let $\psi:G_{\mathbb{F}_q((t))} \to \mathcal{O}_L^\times$ and assume $Im(\psi)\cong \Z/p\Z$. Let 
		$K$ 
		be the fixed field of $\psi^{-1}(1)$ and let $d$ be the largest ramification 
		break (upper numbering) of $G_{K/\mathbb{F}_q((t))}$. Then there exists a $p$-Frobenius 
		structure $E_r$ of 
		$\psi$ such that $E_r \in \mathcal{O}_\mathcal{E}^{d(p-1)}\cap 
		\mathcal{E}^{\leq 0}$ and $E_r \equiv 1 \mod \gothm$. Furthermore,
		if $c\in 1+\gothm \mathcal{O}_\mathcal{E}$ is another $p$-Frobenius structure,
		there exists $b\in 1+\gothm \mathcal{O}_\mathcal{E}$ with $E_r=\frac{b^\nu}{b}c$.
	\end{proposition} 
	
	\begin{remark}
		Since $G_{K/F}\cong \Z/p\Z$, there is only one ramification break. In this case,
		the lower numbering agrees with the upper numbering (see 
		the example at the end of Chapter IV \S 3 in \cite{Serre-local_fields}).
	\end{remark}
	
	\begin{proof}
		The extension $K/\mathbb{F}_q((t))$ is given by an equation $y^p-y=r$, where
		$r \in \mathbb{F}_q((t))$. This $r$ is unique up to addition by an element
		of $(\textbf{Fr}-1)\mathbb{F}_q((t))$ (here $\textbf{Fr}$ is the $p$-th power map).
		In particular, we may take $r$ to be of the form
		\begin{align*}
		r &= \sum_{j=0}^{d} r_{j}t^{-j}.
		\end{align*}
		The equation $y^p-y=r$ also defines a finite \`etale $\mathbb{F}_q[t^{-1}]$-algebra
		$B$ that fits into a commutative diagram:
		\begin{equation*} 
		\begin{tikzcd}
		\Spec(K) \arrow[d]\arrow[r] & \Spec(B)\arrow[d] \\
		\Spec(\mathbb{F}_q((t)))\arrow[r] & \mathbb{P}^1-\{0\}=\Spec(\mathbb{F}_q[t^{-1}])
		.
		\end{tikzcd}
		\end{equation*}
		In particular, $\psi$ extends to a representation $\psi^{ext}:Gal(B/\mathbb{F}_q[t^{-1}]) \to \mathcal{O}_L^\times$. Let $\mathcal{O}_L\ang{t^{-1}}\subset \mathcal{O}_\mathcal{E}$ be the Tate algebra
		in $t^{-1}$ with coefficients in $\mathcal{O}_L$. Note that $\nu$ restricts to a
		$p$-Frobenius endomorphism of $\mathcal{O}_L\ang{t^{-1}}$. All projective modules
		over $\mathcal{O}_L\ang{t^{-1}}$ are free, so that $M_{\psi^{ext}}$ is isomorphic to $\mathcal{O}_L\ang{t^{-1}}$ as a $\mathcal{O}_L\ang{t^{-1}}$-module. We see that 
		$M_{\psi}=M_{\psi^{ext}}\otimes_{\mathcal{O}_{L}\ang{t^{-1}}}\mathcal{O}_{\mathcal{E}}$.
		In particular, any $p$-Frobenius structure of $M_{\psi^{ext}}$ is a $p$-Frobenius
		structure of $M_{\psi}$.
		
		A series $a\in \mathcal{O}_L\ang{t^{-1}}$ is a $p$-Frobenius structure for 
		$M_{\psi^{ext}}$ if and only if for every $k\geq 1$ and $x \in \mathbb{P}^1(\mathbb{F}_{q^k})- \{0\}$
		we have
		\begin{align*}
		\prod_{i=0}^{ak-1} a([x])^{\nu^{i}} &= \psi^{ext}(Frob_x) \\
		&=	\zeta_p^{Tr_{\mathbb{F}_{q^k}/\mathbb{F}_p}(r(x))}, 
		\end{align*}
		where $[x]$ is the Teichmuller lift of $x$. This fact is essentially Chebotarev's density theorem. We see that
		\begin{align*}
		E_r &= \prod_{j=0}^{d} E([r_j] t^{-j} \gamma), 
		\end{align*}
		is a $p$-Frobenius structure of $M_{\psi^{ext}}$. Since $E(x) \in \Z_p[[x]]$, it is clear that $E_r \in 
		\Z_p[[\pi_dt^{-1}]] \subset \mathcal{O}_\mathcal{E}^{d(p-1)}\cap 
		\mathcal{E}^{\leq 0}$. The last part of the Proposition is identical to the
		proof of Proposition \ref{proposition: unramified character proposition}.
	\end{proof}

	\subsection{The $F$-crystal associated to $\rho$}
	We now study the Frobenius structure of the $F$-crystal associated to $\rho$. 
	We continue with the setup from \S \ref{section: global bounds}.
	\subsubsection{The global Frobenius structure}
	Let $\mathcal{L}$ be a rank one $\mathcal{O}_L$-module on which $\pi_1(V)$ acts 
	through 
	$\rho$. Let $C\to X$ be the $\Z/p\Z$-cover that trivializes $\rho$. 
	We 
	have $C \times_X V=\Spec(\overline{R})$. Note that $\overline{R}$ is isomorphic to 
	$\overline{B}[x]/\overline{f}(x)$ where $\overline{f}(x)=x^p-x+b$ for some $b \in \overline{B}$. Let
	$f(x)\in B^\dagger[x]$ be a lift of $\overline{f}$. We define 
	$\widehat{R}=\widehat{B}[x]/f(x)$ and $R^\dagger = B^\dagger[x]/f(x)$. The $F$-crystal
	corresponding to $\rho$ is the $\widehat{B}$-module $M=(\widehat{R} \otimes 
	\mathcal{L})^{Gal(C/X)}$, 
	where the Frobenius structure comes from the extension of $\nu:\widehat{B} \to 
	\widehat{B}$ to an endomorphism of $\widehat{R}$. 
	\begin{proposition}\label{proposition: the global frobenius descends to wc ring}
		The subring $R^\dagger$ of $\widehat{R}$ is preserved by $Gal(C/X)$ and $\nu$.
	\end{proposition}
	\begin{proof}
		Let $g$ be either $\nu$ or an element of $Gal(C/X)$. Let $z \in R^\dagger$ 
		and let
		$u$ be the image of $x$ in $R^\dagger$. Then $z= \sum_{i=0}^{p-1} a_i u^i$, where
		$a_i \in B^\dagger$. Similarly, we write
		$z^g=\sum_{i=0}^{p-1} b_i u^i$, where $b_i \in \widehat{B}$ (here $z^g$
		is the image of $z$ by the action of $g$). Consider a point
		$Q\in W$ and let $P \in C$ be a point in $C \times_X \{Q\}$. This 
		gives rise to an extension $\mathcal{O}_{\mathcal{E}_P^\dagger}$ of 
		$\mathcal{O}_{\mathcal{E}_{Q}^\dagger}$, that is preserved by $g$. Thus 
		$b_i 
		\in \mathcal{O}_{\mathcal{E}_{Q}^\dagger}$, and we see that $b_i \in 
		B^\dagger$ from the Cartesian diagram \eqref{overconvergent expansion 
			diagram}. This proves the 
		proposition.
	\end{proof}
	\begin{corollary} \label{corollary: Frobenius descends to something OC}
		Let $M^\dagger = (R^\dagger \otimes 
		\mathcal{L})^{Gal(C/X)}$. The map $M^\dagger\otimes_{B^\dagger} \widehat{B}\to M$ 
		is a $\nu$-equivariant isomorphism.   
	\end{corollary}
	\begin{lemma} \label{lemma: the $F$-crystal is free}
		The module $M^\dagger$ (resp. $M$) is a free $B^\dagger$ (resp. $\widehat{B}$-module). 
		Furthermore, $M$ has a $p$-Frobenius structure $\alpha_0$ contained in 
		$1 + \gothm B^\dagger$. 
	\end{lemma}
	\begin{proof}
		Let $G=Gal(C/X)$. We claim that $M=\Ext_{\widehat{B}[G]}^1(\widehat{B}, \overline{R})$ is trivial, where
		$G$ acts trivially on $\widehat{B}$. We will view $\Ext_{\widehat{B}[G]}^1(\widehat{B}, \overline{R})$ as
		a $\widehat{B}$-module and show that its support is empty.
		Let $x$ be a maximal ideal of $\widehat{B}$ and let $\widehat{B}_x$
		be the localization at $x$. If $\gothm \not\subset x$, then we see that
		\begin{align*}
			\widehat{B}_x \otimes_{\widehat{B}} M &= \Ext_{\widehat{B}_x[G]}^1(\widehat{B}_x,\widehat{B}_x \otimes_{\widehat{B}} \overline{R})=\{0\}.
		\end{align*}
		Now assume that $\gothm \subset x$. Since $\overline{R}$ is an \`etale $G$-cover of $\overline{B}$,
		we see that we see that $\overline{R} \otimes_{\widehat{B}} \widehat{B}_x$
		is a free $\overline{B}_x[G]$-module of rank one. The Eckmann-Shapiro Lemma gives
		\begin{align*}
			\widehat{B}_x \otimes_{\widehat{B}} \Ext_{\widehat{B}[G]}^1(\widehat{B},\overline{R}) &\cong 
			\Ext_{\widehat{B}_x[G]}^1(\widehat{B}_x,\overline{B}_x[G]) \\
			&\cong \Ext^1_{\widehat{B}_x}(\widehat{B}_x,\overline{B}_x), 
		\end{align*}
		which is trivial since $\widehat{B}_x$ is a free $\widehat{B}_x$-module.
		
		Using the vanishing of $M$, we can prove
		inductively that 
		\[(\widehat{R}/\gothm^{k+1} \otimes_{\mathcal{O}_L} \mathcal{L})^G 
		\to (\widehat{R}/\gothm^{k} \otimes_{\mathcal{O}_L} \mathcal{L})^G\]
		is surjective. Passing to the limit, we see that \[(\widehat{R} \otimes_{\mathcal{O}_L} \mathcal{L})^G 
		\to (\overline{R})^G=\overline{B}\] is surjective. This means
		there exists $z \in \widehat{R}$ with $z \equiv 1 \mod \gothm \widehat{R}$
		contained in $\widehat{R}^{\rho^{-1}}$, the $\rho^{-1}$-isotypical subsace of $\widehat{R}$. 
		Now let $z_0 \in \widehat{R}^{\rho^{-1}}$. Then for some $k\geq 0$ we have 
		$z_0 \equiv \pi_\circ^k a_k \mod \pi_\circ^{k+1}$, where $a_k \in \widehat{B}$. 
		Then $\pi_\circ^k a_kz-z_0$ is in the $\widehat{R}^{\rho^{-1}}$ and is divisible by $\pi_\circ^{k+1}$.
		By continuing this approximation, we see $\widehat{R}^{\rho^{-1}}=\widehat{B}z$,
		which proves the result for $M$.
		To prove $M^\dagger$ is free, we remark that $B^\dagger$
		is a Zariski ring, so that the functor $-\otimes \widehat{B}$ is fully faithful.
		The lemma follows from Corollary \ref{corollary: Frobenius descends to something OC}.
		
	\end{proof}

	\subsubsection{The local Frobenius structures and $U_p$-computations}
	\label{subsubsection: the local frob and Up}
	We fix $\alpha_0$ as in Lemma \ref{lemma: the $F$-crystal is free} and
	$\alpha=\prod_{i=0}^{a-1} \alpha_0^{\nu^i}$. Let $Q \in W$ with $Q=P_{*,i}$. There is a map
	$\overline{B} \to \mathbb{F}_q((u_Q))$, where we expand each function on $V$
	in terms of the parameter $u_Q$. This gives a point  
	$\Spec(\mathbb{F}_q((u_{Q}))) 
	\to V$. By pulling back $\rho$ along this point we obtain
	a local representation $\rho_{Q}: 
	G_{\mathbb{F}_q((u_Q))} \to 
	\mathcal{O}_L^\times$, where $G_{\mathbb{F}_q((u_Q))}$ is the absolute Galois group of 
	$\mathbb{F}_q((u_Q))$. There are three cases we need to consider.
	The first case is when $*=1$. In this case $\rho_Q$ is unramified. This is because $\rho$ is only
	ramified at the points $\tau_1,\dots,\tau_{\mathbf{m}}$ and by Lemma \ref{lemma: map to p1} 
	we have $\eta(\tau_i)\in \{0,\infty\}$. The second case is when $*\in \{0,\infty\}$
	and $\rho_Q$ is unramified. The last case is when $* \in \{0,\infty\}$ and
	$\rho_Q$ is ramified.

	\begin{enumerate}[label=\Roman*.]
		\item If $*=1$, then $\nu_Q$ sends $u_Q \mapsto \sqrt[p-1]{(u_Q^{p-1}+1)^p-1}$ (see
		the end of \S \ref{subsection Global Frobenius and Up}). 
		Since $\rho_Q$ is unramified, we know from Proposition \ref{proposition: unramified character proposition} that 
		there exists $b_Q \in 1 + \gothm\mathcal{O}_{\mathcal{E}_Q}^\dagger$ with
		$c_Q=\frac{b_Q^\nu}{b_Q} \alpha_0 \in 1+\gothm$. 
		We define 
		$\mathcal{M}_Q$ to 
		be 
		a copy of $\mathcal{A}$ inside of $\mathcal{E}_Q^\dagger$ (see \eqref{definition of A} for
		the definition of $\mathcal{A}$). 
		From Proposition \ref{proposition: type 2 Frobenius} we see that:
		\begin{align} 
		\begin{split}\label{UP computation type 1}
		U_p \circ c_Q (p^{a(j+pn)}u_Q^{-(j+pn)}) & 
		\in 
		p^{n}\mathcal{M}_Q, \\
		U_p \circ c_Q (\mathcal{M}_Q) \subset \mathcal{M}_Q. 
		\end{split}
		\end{align}
		\label{Type A Frobenius blurb}
		
		\item Next, consider the case where $*$ is $0$ or $\infty$ and $\rho_{Q}$ 
		is 
		unramified. Then $\nu_Q$ sends $u_Q \mapsto u_Q^p$. As 
		in \ref{Type A Frobenius blurb}, there exists $b_Q \in 
		1+\gothm \mathcal{O}_{\mathcal{E}_Q}^\dagger$ such that 
		$c_Q=\frac{b_Q^\nu}{b_Q} \alpha_0 \in 1+\gothm$. 
		We define $\mathcal{M}_Q$ to be a copy of 
		$\mathcal{O}_{\mathcal{E}}^{\frac{p-1}{p}}$ contained in 
		$\mathcal{E}_Q^\dagger$. By Proposition \ref{type 1 operator property} 
		we have:
		\begin{align}
		\begin{split}\label{UP computation type 2}
		U_p \circ c_Q (\pi^{pn} u_Q^{-n}) &\in p^n\mathcal{M}_Q, \\
		U_p \circ c_Q (\mathcal{M}_Q) &\subset \mathcal{M}_Q.
		\end{split}
		\end{align}
		\item Finally, we consider the case where $*=0$ and $\rho_{Q}$ is ramified. 
		In 
		this case, $Q =\tau_i$ and again $\nu_Q$ sends $u_Q \mapsto u_Q^p$. From Proposition \ref{theorem: 
			ASW frobenius structure}, there exists $b_Q\in 
		1+\gothm \mathcal{O}_{\mathcal{E}_Q}^\dagger$ such that 
		$c_Q=\frac{b_Q^\nu}{b_Q} \alpha_0  \in \mathcal{O}_\mathcal{E}^{d_i(p-1)}$.
		Define $\mathcal{M}_Q$ to be a copy of 
		$\mathcal{O}_\mathcal{E}^{\frac{d_i(p-1)}{p}}$ contained in 
		$\mathcal{E}_Q^\dagger$. 
		By Proposition \ref{type 1 operator property} 
		we have:
		\begin{align}
		\begin{split}\label{UP computation type 3}
		U_p \circ c_Q (\pi_{d_i}^{pn} u_Q^{-n}) &\in 
		\pi_{d_i}^{(p-1)n}\mathcal{M}_Q, \\
		U_p \circ c_Q (\mathcal{M}_Q) &\subset \mathcal{M}_Q.
		\end{split}
		\end{align}
	\end{enumerate}
	
	\subsubsection{Global to semi-local} \label{paragraph: global to semi-local setup}
	We define the following spaces:
	\begin{align*}
	\begin{split} 
	\mathcal{R} &= \bigoplus_{Q \in W} 
	\mathcal{E}_Q,  ~~~~~~
	\mathcal{R}^\dagger = \bigoplus_{Q \in W} 
	\mathcal{E}_Q^\dagger,\\
	\mathcal{S} &= \Bigg(\bigoplus_{\stackrel{*=0,\infty}{i=1}}^{r_*} 
	\mathcal{E}_{P_{*,i}}^{\leq 
		-1} \Bigg) \oplus \Bigg( \bigoplus_{i=1}^{r_1} \mathcal{E}_{P_{*,i}}^{\leq - p} 
	\Bigg ),\\
	\mathcal{S}^{\dagger}&=\mathcal{R}^\dagger \cap \mathcal{S}, \\
	\mathcal{W}&= \bigoplus_{Q\in W} \mathcal{M}_Q \subset \mathcal{R}^\dagger. 
	\end{split}
	\end{align*}
	We define $\mathcal{O}_{\mathcal{R}}$ to be $\bigoplus_{Q \in W} 
	\mathcal{O}_{\mathcal{E}_Q}$ and we define $\mathcal{O}_{\mathcal{R}^\dagger}$ (resp. $\mathcal{O}_{\mathcal{S}}$ and $\mathcal{O}_{\mathcal{S}^\dagger}$) 
	to be $\mathcal{R}^\dagger \cap \mathcal{O}_{\mathcal{R}}$ (resp. $\mathcal{S} \cap \mathcal{O}_{\mathcal{R}}$ and $\mathcal{S}^\dagger \cap \mathcal{O}_{\mathcal{R}}$).
	There is natural projection maps $pr:\mathcal{R} \to \mathcal{S}$, which
	is the direct sum of the projection maps described in \S \ref{subsection: basic definitions}.
	By the definition of each summand of $\mathcal{W}$ we see that 
	\begin{align}\label{equation: projecting onto tails ends up in W}
		\ker(pr)\cap \mathcal{O}_\mathcal{R} &\subset \mathcal{W}.
	\end{align}
	We
	may view $\widehat{\mathcal{B}}$ (resp. $\mathcal{B}^\dagger$) as a subspace of $\mathcal{R}$ (resp. $\mathcal{R}^\dagger)$ from \eqref{overconvergent expansion diagram}. Let $\vec{c}$ (resp. 
	$\vec{b}$) denote the element of $\mathcal{R}$ whose $Q$-coordinate is
	$c_Q$ (resp. $b_Q$). This gives an operator $U_p \circ \vec{c}: \mathcal{R}^\dagger \to \mathcal{R}^\dagger$ and we have
	\begin{align} \label{equation: W is preserved}
		U_p \circ \vec{c} (\mathcal{W}) \subset \mathcal{W}.
	\end{align}
	Note that
	\begin{align}\label{local change of Frobenius eq}
	\frac{\vec{b}^\nu}{ \vec{b}} \alpha_0 &= \vec{c},
	\end{align}
	where multiplication is done coordinate-wise. We remark that
	\begin{align} \label{local change of frob is eq 1 mod p}
	\vec{b} &\equiv 1\mod \gothm.
	\end{align}

	\section{Normed vector spaces and Newton polygons}
	\label{section: newton polygons and functional analysis}
	In this section we study Newton polygons of operators 
	on normed vector spaces over $L$. We refer the reader to 
	\cite{Serre-p-adic_banach} or \cite{Monsky-forma_cohomology3}
	for many standard facts on Fredholm determinants in the $p$-adic setting.
	In an effort to make this article self-contained, we provide specific references when
	necessary.
	\subsection{Normed vector spaces and Banach spaces}
	Let $V$ be a vector space over $L$ with a norm $|\cdot|$
	compatible with the $p$-adic norm $|\cdot |_p$ on $L$. We
	will make the assumption that for every $x \in V$, the norm $|x|$
	lies in $|L|_p$, the norm group of $L$. We say that $V$
	is a \emph{Banach} space if it is also complete. Let $V_0\subset V$ 
	denote the subset consisting of $x \in V$ satisfying $|x|\leq 
	1$
	and let $\overline{V}=V_0/\gothm V_0$. If $W$ is a subspace 
	of $V$, we will
	automatically give $W$ the subspace norm. 
	\begin{definition}
		Let $I$ be a set. We let $\mathbf{s}(I)$ denote the set of families
		$x=(x_i)_{i\in I}$, with $x_i \in L$, such that $|x|=\sup\limits_{i\in I}|x_i|_p < \infty$. 
		Then $\mathbf{s}(I)$ is a Banach space with the norm $|\cdot|$. We let
		$\mathbf{c}(I) \subset \mathbf{s}(I)$ denote the subspace of families $x$
		with $\lim\limits_{i\in I} x_i = 0$.  It is convenient
		to represent $x$ as a sum
		\begin{align*}
			x= \sum_{i \in I} x_i e_i,
		\end{align*}
		where $e_i \in \mathbf{s}(I)$ is the family with $1$ in the $i$-coordinate
		and $0$ in the other coordinates. 
	\end{definition}
	
	\begin{definition}
		An \emph{integral basis} of $V$ is a subset $B(I)=\{e_{i}\}_{i \in I} \subset V$
		such that every $x \in V$
		can be written uniquely as 
		\begin{align}  \label{orthonormal basis definition}
		x &= \sum_{i \in I} x_ie_i,
		\end{align}
		with $|x|=\sup\limits_{i\in I} |x_i|_p$.
		In particular, we may regard $V$ as a subspace of $\mathbf{s}(I)$. 
	\end{definition}
	\begin{definition}
		An \emph{orthonormal basis} of $V$ is an integral basis $B(I)=\{e_{i}\}_{i \in I} \subset V$ with the additional
		assumption that for each $x = \sum_{i \in I} x_ie_i \in V$
		we have $\lim\limits_{i\in I} x_i = 0$. In particular, we may regard
		$V\subseteq \mathbf{c}(I)$, and this inclusion is equality if $V$ is a Banach space. By \cite[Proposition I]{Serre-p-adic_banach},
		every Banach space over $L$ has an orthonormal basis, and thus every Banach space 
		is of the form $\mathbf{c}(I)$. 
	\end{definition}
	
	\begin{example} \label{example: banach space stuff}
		Consider the Banach space $V=\mathcal{O}_L[[t]]\otimes \Q_p$. The set $\{t^n\}_{n \in \Z_{\geq 0}}$
		is an integral basis of $V$ and gives an isomorphism $V \cong \mathbf{s}(\Z_{\geq 0})$.
		From \cite[Lemme I]{Serre-p-adic_banach} any orthonormal basis
		of $V$ reduces to an $\mathbb{F}_q$-basis of $\overline{V}=\mathbb{F}_q[[t]]$,
		and thus must be uncountable.
		The Tate algebra $L\ang{t} \subset V$ is a Banach space,
		which we identify with $\mathbf{c}(\Z_{\geq 0})$. The
		ring of overconvergent functions $L\ang{t}^\dagger$ has $\{t^n\}_{n \in \Z_{\geq 0}}$
		as an orthonormal basis, but is not a
		Banach space as it is not $p$-adically complete.  
	\end{example}
	
	\begin{definition} \label{defintiion: expanding the radius}
		Let $b=(b_i)_{i \in I} \in \mathbf{c}(I)$. We define $\mathbf{s}(I,b) \subset \mathbf{s}(I)$ 
		to be the subspace consisting of families of the form $(x_ib_i)_{i \in I}$
		where $\sup\limits_{i \in I} | x_i|_p<\infty$. Similarly, we define 
		$\mathbf{c}(I,b) \subset \mathbf{s}(I,b)$ to be the subspace consisting of families 
		$(x_ib_i)_{i \in I}$ where $\lim\limits_{i \in I} x_i = 0$. 
		Note that $B(I,b)=\{b_ie_i\}_{i \in I}$ is an integral (resp. orthonormal) basis of
		$\mathbf{s}(I,b)$ (resp. $\mathbf{c}(I,b)$). 
	\end{definition}
	
	\begin{definition} \label{definition: partial ordering}
		We define a partial ordering on $\mathbf{c}(I)$ as follows: for $x=(x_i)_{i \in I},y=(y_i)_{i \in I} \in \mathbf{c}(I)$ we have $x > y$ if $\lim\limits_{i \in I} v_p(y_i)-v_p(x_i)=\infty$.
		Note that if $x>y$ we have $\mathbf{c}(I,y) \subset \mathbf{c}(I,x)$ and $\mathbf{s}(I,y) \subset \mathbf{s}(I,x)$.
	\end{definition}
	
	\begin{example}
		Continuing with Example \ref{example: banach space stuff}, let $m$
		be a rational number and assume
		$\pi_m \in \mathcal{O}_L$.  Consider $y=(\pi_m^n)_{n \in \Z_{\geq 0}} \in \mathbf{c}(\Z_{\geq 0})$.
		Then $\mathbf{c}(\Z_{\geq 0}, y)$ (resp. $\mathbf{s}(\Z_{\geq 0}, y)$)
		corresponds to $L\ang{\pi_m t}$ (resp. $L[[\pi_mt]]$), which are 
		analytic functions on the disc $v_p(x)\geq -\frac{1}{m(p-1)}$ (resp.
		bounded analytic functions on the disc $v_p(x)>- \frac{1}{m(p-1)}$). 
	\end{example}
	\subsubsection{Restriction of scalars to $E$}
	Let $I$ be a set and assume that $V \subset \mathbf{s}(I)$ has $B(I)$
	as an integral basis. We may regard $V$ as a vector space over $E$. Let
	$\zeta_1=1, \zeta_2,\dots,\zeta_a \in \mathcal{O}_L$ be elements
	that reduce to a basis of $\mathbb{F}_q$ over $\mathbb{F}_p$ modulo $\pi_\circ$. Let 
	$I_E=I \times \{1,\dots, a\}$. We define 
	\begin{align*}
	B(I_E) &= \{ \zeta_j e_i\}_{(i,j)\in I_E}.
	\end{align*}
	Note that $B(I_E)$ is an integral basis of $V$ over $E$. 
	We let $\mathbf{s}(I_E)$ (resp. $\mathbf{c}(I_E)$) denote the 
	set of families
	$x=(x_{(i,j)})_{(i,j)\in I_E}$, with $x_{(i,j)} \in E$, such that $|x|=\sup\limits_{(i,j)\in I_E}|x_{(i,j)}| < \infty$
	(resp. $\lim\limits_{(i,j)\in I_E} x_{(i,j)} = 0$). 
	For $x \in \mathbf{c}(I_E)$, we define 
	$\mathbf{c}(I_E,x)$ and $\mathbf{s}(I_E,x)$ analogous to Definition \ref{defintiion: expanding the radius} and we define the partial ordering on $\mathbf{c}(I_E)$ analogous to Definition \ref{definition: partial ordering}. Note that $\mathbf{s}(I_E)$ (resp. $\mathbf{c}(I_E)$) can be naturally identified 
	with $\mathbf{s}(I)$ (resp. $\mathbf{c}(I_E)$) as normed vector spaces over $E$ and
	this identification respects the two partial orderings. It will be convenient to start with $x \in \mathbf{c}(I)$ and then consider
	$\mathbf{c}(I_E,x)$ using this identification.
	
	\subsection{Generalities about Newton polygons}
	Let $\alpha \in \mathbb{R}_{\geq 0} \cup \infty$. 
	A \emph{polygonal segment} $P$ of length $\alpha$ is
	a graph of points $(x,f(x))$ where $x \in [0,\alpha]$
	and $f:[0,\alpha] \to \mathbb{R}$ is a continuous piecewise
	linear function. We say that a polygonal segment $P$ is a \emph{Newton polygon}
	if $\alpha \in \Z_{\geq 0} \cup \infty$
	and $f$ satisfies the following properties:
	\begin{enumerate}
		\item $f(0)=0$
		\item For any integer $i\in [0,\alpha)$,
		the function $f(x)$ is linear on the domain $x\in [i,i+1]$
		with slope $m_i\geq 0$.
		\item The $m_i$ are nondecreasing, i.e., $m_{i+1}\geq m_i$ for $0\leq i < 
		\alpha$.
		\item If $\alpha=\infty$, then the $m_i$ are unbounded, i.e. 
		$\lim\limits_{i\to\infty} 
		m_i=\infty$.
	\end{enumerate}
	We will refer to the multiset $\{m_i\}_{i \in I}$ as the \emph{slope-set} of $P$
	and its elements as the slopes of $P$.
	Note that the slope-set of a Newton polygon $P$ determine $P$ entirely. In 
	particular,
	let $I$ be a countable set and let $N=\{n_i\}$
	be a multiset of nonnegative numbers indexed by $I$. Assume that
	for any $r>0$, the subset 
	\begin{align*}
	N_r &= \{ n \in N ~|~n<r\}
	\end{align*}
	is finite. Then there exists a unique Newton polygon
	$P_N$ whose slope-set is $N$.
	
	Let us now describe some operations on the set of Newton polygons.
	If $P$ and $P'$ are Newton polygons whose slope-sets
	are $N$ and $N'$, we define the
	concatenation of $P$ and $P'$ to be the Newton polygon
	\[ P \sqcup P' = P_{N \sqcup N'}. \]
	That is, $P \sqcup P'$ is the Newton polygon whose slope-set
	is the disjoint union of $N$ and $N'$. Next, let $r>0$.
	We define the $r$-truncation $P_{<r}$ of $P$ to be the Newton polygon
	whose slope-set is $N_r$. Note that $P_{<r}$ necessarily
	has finite length. We may also scale Newton polygons.
	Let $c=\frac{a}{b}$, where $a$ and $b$ are coprime natural numbers. We define
	\[cP = \{ (cx,cy) \in \mathbb{R}^2 ~|~ (x,y) \in P\}.\]
	Note that $cP$ is only a polygonal segment in general.
	It is a Newton polygon if and only if
	the multiplicity of every slope of $P$ is a multiple of $b$.
	
	We may compare two polygonal segments as follows.
	For $i=1,2$, let $P_i$ be a polygonal segments of length $\alpha_i$
	determined by the function $f_i:[0,\alpha_i] \to \mathbb{R}$. 
	Then we write 
	\[ P_1 \succeq P_2 \]
	if and only if $f_1(x) \geq f_2(x)$ for all $0 \leq x \leq 
	\min(\alpha_1,\alpha_2)$. 
	When $P_1 \succeq P_2$, we say that $P_1$ lies above $P_2$ (indeed,
	on the $xy$-plane $P_1$ does lie above $P_2$ where both are defined).
	If $P_2$ is a Newton polygon with slope-set $N$, we will occasionally write 
	$P_1\succeq N$ instead of $P_1 \succeq P_2$.

	Finally, we introduce Newton polygons associated to power series. 
	Let $Q(s)=\sum\limits_{n=0}^\infty 
	a_n 
	s^n \in \mathcal{O}_L[[s]]^\times$ with $\lim v_p(a_n)=\infty$. Let $\alpha=\deg(Q)$
	if $Q$ is a polynomial and $\infty$ otherwise. Then
	for $*$ equal to $p$ or $q$, we define $NP_*(Q(s))$ to be the length $\alpha$ 
	Newton polygon that is the lower convex hull of the points $(n,v_*(a_n))$.
	We have the following relations:
	\begin{align} \label{Newton polygon relations}
	\begin{split}
	NP_*(Q(s^a)) &= \{ (ax,y) ~|~(x,y) \in NP_*(Q(s))\} \\
	NP_p(Q(s)) &= \{(x,ay)~|~(x,y) \in NP_q(Q(s))\}.
	\end{split}
	\end{align}

	\subsection{Characteristic series and Fredholm determinants}
	\subsubsection{Nuclear operators}
	Let $V$ be a vector space over $L$ and let $u:V \to V$ (resp. $v: V \to V$) be an
	$L$-linear (resp. $E$-linear) operator. The definition of
	a nuclear operator is slightly cumbersome and since we won't use
	the definition, we content ourselves with an informal definition (see \cite[\S 1]{Monsky-forma_cohomology3}):
	\begin{definition}
		An operator is \emph{nuclear} if for any $r\in \R$, there
		are only finitely many generalized eigenvalues (counting multiplicity) with
		$p$-adic valuation less than $r$.
	\end{definition}
 	Assume that $u$ (resp. $v$) are nuclear operators. Monsky 
 	associates an entire power series $P(u,s)$  (resp. $P(v,s)$) contained in $1 + K[[s]]$
 	to $u$ (resp. $v$). The zeros of this power series are the reciprocal generalized eigenvalues
 	of $u$ (resp. $v$) counted with multiplicity. In particular, $P(u,s)$ and $P(v,s)$ act
 	as a ``characteristic series'' of $u$ and $v$. Note that
 	if $V$ is finite, then $u$ (resp. $v$) is nuclear and $P(u,s)=\det(1-su)$ (resp. $P(v,s)=\det(1-sv)$).

	\subsubsection{Fredholm determinants}
	Let $V$ be a vector space over $L$ with norm $|\cdot|$ and
	let $B(I)=\{e_{i}\}_{i \in I}$
	be an integral basis of $V$. We assume that $I$ is countable.
	Let $u:V \to V$ (resp. $v: V \to V$) be an
	$L$-linear (resp. $E$-linear) operator. Let $(n_{i,j})$
	be the matrix of $u$ with respect to 
	$B(I)$. The \emph{Fredholm determinant} of $u$ with respect to $B(I)$ is defined to be
	\begin{align} \label{Fredholm definition}
	\begin{split}
	\det(1-su| B(I)) &= \sum_{n=0}^\infty c_ns^n \\
	c_n &= (-1)^n \sum_{\stackrel{S \subset I}{|S|=n}} 
	\sum_{\sigma \in \Sym(S)} \text{sgn}(\sigma) \prod_{i \in S} n_{i,\sigma(i)}.
	\end{split}
	\end{align}
	We define the Fredholm determinant $\det(1-sv|B_E(I))$ in
	an analogous manner using the matrix of $v$ with respect to $B_E(I)$.
	
	\begin{remark}
		In general, the Fredholm determinant is not defined, and if it is
		defined it may depend on the basis. However, when we restrict to certain types of operators,
		the Fredholm determinant is defined and independent of basis, and thus is 
		an intrinsic invariant of the operator.
	\end{remark}

	For the remainder of this subsection, we restrict our attention to 
	the $L$-linear operator $u$. We remark that all analogous results hold true for 
	$v$. 

	\begin{definition}
		For $i \in I$, we define $\textbf{row}_i(u,B(I))=\inf\limits_{j \in I} v_p(n_{i,j})$
		and $\textbf{col}_i(u,B(I))=\inf\limits_{j \in I} v_p(n_{j,i})$. That is,
		$\mathbf{row}_i(u,B(I))$ (resp. $\mathbf{col}_i(u,B(I))$) is the smallest $p$-adic
		valuation that occurs in the $i$-th row (resp. column) of the matrix of $u$.
		Note that $\textbf{col}_i(u,B(I))=\log_p|u(e_i)|$.
	\end{definition}

	\begin{definition} 
		Assume that $V=\mathbf{c}(I)$. We say that $u$ is \emph{completely continuous} if it
		is the $p$-adic limit of $L$-linear operators 
		whose image is finite dimensional. Equivalently, $u$
		is completely continuous if $\lim\limits_{i \in I} \mathbf{row}_i(u,B(I)) = \infty$
		(see \cite[Theorem 6.2]{Monsky-padic_notes}).
	\end{definition}
	
	\begin{definition}
		Assume that $V=\mathbf{s}(I)$. We say that $u$ is \emph{tight}
		if $\lim\limits_{i \in I} \mathbf{col}_i(u,B(I)) = \infty$.
	\end{definition}
	
	\begin{proposition} \label{proposition: completely continuous  operators have good Fredholm determinant}
		Let $V$ be a Banach space over $L$ with a countable orthonormal basis $B(I)$ and let $u$ be a completely continuous operator
		on $V$. The Fredholm determinant $\det(1-su| B(I))$ does not depend on
		the choice of orthonormal basis and $\det(1-su|B(I))=P(u,s)$. 
	\end{proposition}
	\begin{proof}
		This is due to Serre. See \cite[Theorem 1.3]{Monsky-forma_cohomology3}
		or \cite[Chapter 6]{Monsky-padic_notes}.
	\end{proof}
	
	If $V$ is not a Banach space, we may still compute $P(u,s)$
	using Fredholm determinants if $V$ is the union of Banach spaces
	on which $u$ restricts to a completely continuous  operator.
	\begin{lemma} \label{lemma: expanding radius means completely continuous  operator}
		Let $x, y \in \mathbf{c}(I)$ with $x>y$. If $u(\mathbf{c}(I,x)) \subset \mathbf{s}(I,y)$,
		then $u$ restricts to a completely continuous  operator on $\mathbf{c}(I,x)$. 
	\end{lemma}
	\begin{proof}
		This follows from the definition of our partial ordering and the
		definition of completely continuous.
	\end{proof}

	\begin{lemma} \label{lemma: as long as operator is defined radius doesn't matter}
		Let $x=(x_i)_{i \in I},y=(y_i)_{i \in I}  \in \mathbf{c}(I)$ with $x>y$. Assume that $u$ restricts to a completely continuous 
		operator on $\mathbf{c}(I,x)$ and that $u(\mathbf{s}(I,y)) \subset \mathbf{s}(I,y)$.
		Then $\det(1-su|B(I,y))$ exists and is equal to $\det(1-su|B(I,x))$.
	\end{lemma}
	\begin{proof}
		Let $(n_{i,j})$ be the matrix of $u$ with respect to $B(I,x)$. The matrix
		of $u$ with respect to $B(I,y)$ is $(n_{i,j}\frac{x_jy_i}{x_iy_j})$. 
		For a finite subset $S \subset I$ and $\sigma \in \Sym(S)$ we have
		\begin{align*}
			\prod_{i \in S} n_{i,\sigma(i)} &= \prod_{i \in S} n_{i,\sigma(i)}\frac{x_{\sigma(i)}y_i}{x_iy_{\sigma(i)}},
		\end{align*} 
		and the result follows from \eqref{Fredholm definition}.
	\end{proof}
	\begin{corollary} \label{corollary: union of completely continuous  }
		Let $S$ be a set and let $\{x_m\}_{m \in I}$ be a family of elements in $\mathbf{c}(I)$.
		Assume that $V=\cup_{m \in S} \mathbf{c}(I,x_m)$ and that $u$ restricts
		to a completely continuous  operator on $\mathbf{c}(I,x_m)$ for each $m \in S$.
		Then $u$ is a nuclear operator on $V$. Furthermore, let $x \in \mathbf{c}(I)$
		and assume $u(\mathbf{s}(I,x)) \subset \mathbf{s}(I,x)$. Then
		$\det(1-su|B(I,x))$ exists and is equal to $P(u,s)$.
	\end{corollary}
	\begin{proof}
		This follows from \cite[Theorem 1.6]{Monsky-forma_cohomology3}
		and Lemma \ref{lemma: as long as operator is defined radius doesn't matter}.
	\end{proof}
	
		\subsubsection{Newton polygons of operators}
	Let $V$ be vector space over $L$ with norm $|\cdot|$ and let $B(I)=\{e_i\}_{i \in I}$ be 
	an orthonormal basis indexed by $I$. Let $u:V \to V$ (resp. $v:V \to V$) be
	an $L$-linear (resp. $E$-linear) operator that is nuclear.
	We define the Newton polygon of $u$ (resp. $v$) to be $NP(u)=NP(P(u,s))$ (resp.  $NP(v)=NP(P(v,s))$).
	We obtain estimates for $NP(v)$ by estimating the columns of the matrix representing $v$.
	\begin{lemma} \label{lemma: estimating NP by estimating columns}
		Let $S$ be a set and let $\{x_m\}_{m \in S}$ be a family of elements in $\mathbf{c}(I)$.
		Assume that $V=\cup_{m \in S} \mathbf{c}(I_E,x_m)$ and that $v$ restricts
		to a completely continuous operator on each $\mathbf{c}(I_E,x_m)$.
		We also assume there exists $x \in \mathbf{c}(I)$ such that $v$ restricts to
		a tight operator on $\mathbf{s}(I_E,x)$. 
		If $v$ is $\nu^{-1}$-semilinear, then
		\begin{align*}
		NP(v) \succeq  \{\mathbf{col}_{(i,1)}(v,B(I_E,x))\}^{\times a}_{i \in I}
		\end{align*}
	\end{lemma}
	\begin{proof}
		Since $v$ is tight, we know that $\{\mathbf{col}_{(i,j)}(v,B(I_E,x))\}_{(i,j) \in I_E}$
		is the slope-set of a Newton polygon. We claim
		\begin{align}\label{eq: estimating via columns}
		NP(v) \succeq \{\mathbf{col}_{(i,j)}(v,B(I_E,x))\}_{(i,j) \in I_E}.
		\end{align}
		To see this, consider the sum that defines $c_n$ in \eqref{Fredholm definition}.
		Each product $\prod  n_{i,\sigma(i)}$ is taken over elements from $n$ distinct columns.
		Thus, the $p$-adic valuation of $c_n$ is at least the sum of the 
		$n$ smallest elements of $\{\mathbf{col}_{(i,j)}(v,B(I_E,x))\}_{(i,j) \in I_E}$.
		This implies
		\begin{align*}
		NP(\det(1-sv|B(I_E,x))) \succeq \{\mathbf{col}_{(i,j)}(v,B(I_E,x))\}_{(i,j) \in I_E},
		\end{align*}
		and \eqref{eq: estimating via columns} follows from Corollary \ref{corollary: union of completely continuous  }.
		To prove the lemma, note that $\mathbf{col}_{(i,j)}(v,B(I_E,x))$ is independent of
		$j$, since $v(\zeta_je_i)=\zeta_j^{\nu^{-1}}v(e_i)$ and $\zeta_j^{\nu^{-1}}$
		is a unit. 
	\end{proof}

	\begin{remark} \label{remark:tightness on first coordinate}
		Let $v$ be $\nu^{-1}$-semilinear. 
		From the end of the proof of Lemma \ref{lemma: estimating NP by estimating columns},
		we also see that $v$ is tight if $\lim\limits_{i \in I} \mathbf{col}_{(i,1)}(v,B(I_E,x))=\infty$.
	\end{remark}

	\subsection{Dwork operators}  \label{section: Dwork operators}
	We now introduce a class of operators known as Dwork operators. These operators were studied by Monsky in \cite{Monsky-forma_cohomology3}, although in a slightly different setting. 
	\subsubsection{Local and semi-local Dwork operators}
\label{subsubsection: Local and semi-local Dwork operators}
	\begin{definition}
		Let $\nu$ be a $p$-Frobenius endomorphism of $\mathcal{O}_{\mathcal{E}^\dagger}$ and let $\sigma=\nu^{a}$.
		We say that an $L$-linear operator $u:\mathcal{E}^\dagger \to \mathcal{E}^\dagger$ (resp. $E$-linear operator $v:\mathcal{E}^\dagger \to \mathcal{E}^\dagger$) is a $q$-\emph{Dwork operator} (resp. $p$-\emph{Dwork operator}) 
		if it is $ \sigma^{-1}$-semilinear (resp. $\nu^{-1}$-semilinear). That is,
		we have 
		$u(y^\sigma x) = yu(x)$ (resp. $v(y^\nu x) = yv(x)$) for all $y \in \mathcal{E}^\dagger$ and $x \in \mathcal{E}^\dagger$.
	\end{definition}

	\begin{lemma} \label{lemma: Dwork growth}
		Let $v: \mathcal{E}^\dagger \to \mathcal{E}^\dagger$ be a $p$-Dwork operator.
		For $m$ sufficiently large we have $v(\mathcal{O}_{\mathcal{E}}^m \otimes \Q_p) \subset \mathcal{O}_{\mathcal{E}}^{\frac{m}{p}} \otimes \Q_p$. 
	\end{lemma}
	
	\begin{proof}
		Let $m$ be large enough so that $t^\nu \in \mathcal{O}_\mathcal{E}^m(p)$ and
		$v(t^i) \in \mathcal{O}_{\mathcal{E}}^\frac{m}{p} \otimes \Q_p$ for
		$i=0, \dots, p-1$. The result then follows from Lemma \ref{lemma: decomposition of frobenius action}.
	\end{proof}
	
	Recall the spaces $\mathcal{R}=\bigoplus_{Q \in W} \mathcal{E}_Q$ and $\mathcal{R}^\dagger=\bigoplus_{Q \in W} \mathcal{E}_Q^\dagger$ introduced in \S \ref{paragraph: global to semi-local setup}.
	We give $\mathcal{R}$ a norm by giving each summand the Gauss norm on $\mathcal{E}_Q$.
	In particular, $\mathcal{R}_0=\mathcal{O}_{\mathcal{R}}$. 
	 Let $V$ be a subspace contained in $\mathcal{R}^\dagger$ and endow $V$ with
	 the subspace norm. 
	 
	 \begin{definition}
	 	We say that an $L$-linear operator $u:\mathcal{R}^\dagger \to \mathcal{R}^\dagger$ (resp. $E$-linear operator $v:\mathcal{R}^\dagger \to \mathcal{R}^\dagger$) is a $q$-\emph{Dwork operator} (resp. $p$-\emph{Dwork operator}) if it is the direct sum of $q$-Dwork operators (resp. $p$-Dwork operators)
	 	and each summand $\mathcal{E}_Q^\dagger$. We say that $u:V \to V$ (resp. $v: V \to V$)
	 	is a $q$-\emph{Dwork operator} (resp. $p$-\emph{Dwork operator}) if
	 	it is the restriction of a $q$-Dwork operator (resp. $p$-Dwork operator) on $\mathcal{R}^\dagger$. 
	 \end{definition}
 
 	\subsubsection{Dwork operators are nuclear} \label{subsubsection: Dwork operators are nuclear}
	Recall the definitions of $\mathcal{S}$, $\mathcal{S}^\dagger$, and $\mathcal{W}$ from \S \ref{paragraph: global to semi-local setup}, as well as the projection map $pr:\mathcal{R} \to \mathcal{S}$.
	We assume $V$ satisfies the following condition:
	\begin{align} \label{condition: surjective on tails with finite coker}
	\textsl{We have } pr(V_0)=\mathcal{O}_{\mathcal{S}^\dagger} \textsl{ and
		$\ker(pr:V \to \mathcal{S}^\dagger)$ is finite dimensional}.
	\end{align} 
	\noindent This implies that $V$ surjects onto $\mathcal{S}^\dagger$.
	We will show that Dwork operators on $V$ are nuclear and become
	completely continuous when restricted to a small enough annulus around each $Q$. This
	is analogous to \cite[Theorem 2.1 and Lemma 2.5]{Monsky-forma_cohomology3}, where
	Monsky shows that Dwork operators on $n$-dimensional ``overconvergent'' affine space
	become completely continuous when restricted to a larger radius of convergence. 
	The first step is to choose an orthonormal
	basis of $V$ that is reasonable to compute with. Define $\mu: W \to \mathbb{N}$
	by
	\begin{align*}
	\mu(P_{*,i}) &= \begin{cases}
	1 & *=0,\infty \\
	p & * = 1.
	\end{cases}
	\end{align*}
	Then define $J \subset \mathbb{N} \times W$ by
	\begin{align} \label{J definition}
	J &= \{(n,Q) ~~ | ~~ n\geq \mu(Q)\}.
	\end{align}
	The set $\{u_Q^{-n}\}_{(n,Q) \in J}$ is an orthonormal basis 
	for $\mathcal{S}^\dagger$ (here we identify $u_Q^{-n}$ with the
	element of $\mathcal{S}^\dagger$ with $u_Q^{-n}$ in the $Q$-th coordinate and $0$
	in the other coordinates). Let $K$ be a set with $\dim_L(\ker_L(pr|_V))$ elements and set
	$I=J \sqcup K$. 
	For $i=(n,Q) \in J$, choose an element $e_i \in V_0$ 
	with $pr(e_i)=u_Q^{-n}$. This exists by \eqref{condition: surjective on tails with finite coker}.
	We also choose an orthonormal basis $\{e_i\}_{i \in K}\in V_0$ of $\ker_L(pr|_V)$
	indexed by $K$. Then $B(I)=\{e_i\}_{i \in I}$ is an orthonormal basis of $V$. 
	By \eqref{equation: projecting onto tails ends up in W} there exists $c_i \in \mathcal{W}$
	for each $i \in I$ with
	\begin{align} \label{eq how the basis looks}
	e_i =\begin{cases}
	u_Q^{-n} + c_{i} & i=(n,Q) \in J \\
	c_i & i \in K.
	\end{cases} 
	\end{align}
	
	\begin{definition}
		For $m > 0$, define a family $(b_{m,n})_{\mathbb{N}} \subset \mathcal{O}_L$
		such that $\{\dots, t^2,t^1,1,b_{m,1}t^{-1}, b_{m,2}t^{-2}, \dots\}$
		is an integral basis of $\mathcal{O}_{\mathcal{E}}^{m}\otimes \Q_p$ (e.g.
		if $m=p-1$ we may take $b_{m,n}=\pi^n$).
		We then define $x_m =\{x_{m,i}\}_{i \in I} \in \mathbf{c}(I)$ by
		\begin{align*}
		x_{m,i}&= \begin{cases}
		1 & i \in K \\
		b_{m,n} & i=(n,Q) \in J,
		\end{cases}
		\end{align*}
		and we set $V_m=\mathbf{c}(I,x_m)$. Note that $V_m$
		consists of the elements of $V$ whose $Q$-th coordinate
		are bounded
		analytic functions on the annulus $0<v_p(u_Q)\leq  \frac{1}{m}$. 
	\end{definition}
	
	\begin{lemma} \label{lemma: results about smaller radii on $V$}
		The following hold:
		\begin{enumerate}[label=\roman*.]
			\item For $m_2>m_1$, we have $V_{m_1}\subset V_{m_2}$.
			\item For each $m>0$ we have $\mathbf{s}(I,x_m) = V \cap (\bigoplus_{Q \in W}\mathcal{O}_{\mathcal{E}}^m\otimes \Q_p)$.
			\item We have
			\[V=\bigcup_{m>0} \mathbf{s}(I,x_m)= \bigcup_{m>0} V_m.\]
		\end{enumerate}
	\end{lemma}
	\begin{proof}
		This follows from the definition of $V_m$. 
	\end{proof}
	\begin{proposition} \label{corollary: completely continuous  on small radius}
		Let $V$ be a subspace of $\mathcal{R}^\dagger$ and assume that
		$V$ satisfies condition \eqref{condition: surjective on tails with finite coker}. 
		Let $v:V \to V$ be a $p$-Dwork operator and let $u=v^a$. 
		For $m$ sufficiently large, $v$ and $u$ restrict to completely continuous operators on the Banach space
		$V_m$. In particular, $v$ and $u$ are nuclear operators. Furthermore,
		if $x \in \mathbf{c}(I)$ and $v(\mathbf{s}(I,x)) \subset \mathbf{s}(I,x)$,
		the Fredholm determinant $\det(1-sv|B(I_E,x))$ exists and is equal to $P(v,s)$.
	\end{proposition}
	
	\begin{proof}
		Note that $v$ restricts to a Dwork operator for $\nu_Q$ on the $Q$-th coordinate of $\mathcal{R}^\dagger$. By Lemma \ref{lemma: Dwork growth} and Lemma \ref{lemma: results about smaller radii on $V$} there exists $m$ such that $v(V_m) \subset \mathbf{s}(I,x_{\frac{m}{p}})$. It is clear that $u(V_m) \subset \mathbf{s}(I,x_{\frac{m}{p}})$ as well.
		As $x_m>x_{\frac{m}{p}}$ it follows from Lemma \ref{lemma: expanding radius means completely continuous  operator}
		that $u$ and $v$ are completely continuous  when restricted to $V_m$.
		The proposition follows from Corollary \ref{corollary: union of completely continuous  }.
	\end{proof}

	\subsubsection{Computing Newton polygons using $a$-th roots}
	When estimating $NP_q(u)$ for a $q$-Dwork operator $u$ on $V$, it is convenient
	to work with a $p$-Dwork operator $v$ that is an $a$-th root of $u$. The reason
	is that we can translate $p$-adic bounds on $P(v,s)$ to $q$-adic bounds on $P(u,s)$.

	\begin{lemma} \label{lemma: L-function via U_p operator lemma}
		Assume that $V$ satisfies condition \eqref{condition: surjective on tails with finite coker}. 
		Let $v$ be a $p$-Dwork operator on $V$ and let $u=v^a$. 
		We further assume that $P(u,s)$ has coefficients in $E$ (a priori its
		coefficients could lie in $L$).
		Then
		$NP_q(u)=\frac{1}{a}NP_p(v)$.
	\end{lemma}
	
	\begin{proof}
		This is a slight modification
		of an argument due to Dwork (see \cite[\S 7]{Dwork-hypersurface_2} or
		\cite[Lemma 2]{Bombieri-exponential_sums}). 
		By Proposition \ref{corollary: completely continuous  on small radius},
		we may choose $m$ large enough so that $u$ and $v$ are completely continuous  completely continuous  when
		restricted to $V_m$. Thus, by  Proposition \ref{proposition: completely continuous  operators have good Fredholm determinant} it suffices to prove the result for the Fredholm determinants of $u$ and $v$ on $V_m$. 
		Recall that 
		$\det(1-su|B(I_E,x_m))$ denotes the Fredholm determinant of
		$u$ on $V_m$ viewed as an $E$-linear operator. Then we have
		\begin{align*}
		\det(1-su|B(I_E,x_m)) &= N_{L/E}\det(1-su|B(I,x_m)),
		\end{align*}
		where $N_{L/E}$ denotes the norm from $L$ to $E$.
		By assumption $\det(1-su|B(I,x_m))$ has coefficients in 
		$E$. This implies
		\begin{align}
			\det(1-su|B(I_E,x_m)) &= \det(1-su|B(I,x_m))^a.\label{Up reduction eq1}
		\end{align}
		Since $u=v^a$, we know by \eqref{Up reduction eq1} that
		\begin{align*}
		\det(1-s^au|B(I,x_m))^a &= 
		\prod_{\zeta^a=1} 
		\det(1-\zeta^a sv|B(I_E,x_m)).
		\end{align*}
		Each term in the product has the same Newton polygon, which means
		\begin{align*}
		NP_q(\det(1- sv|B(I_E,x_m)))&= NP_q(\det(1-s^a u|B(I,x_m))). 
		\end{align*}
		The lemma follows from \eqref{Newton polygon relations}.
	\end{proof}

	\section{Finishing the proof of Theorem \ref{main theorem}}

	\label{section: proof of theorems}

	\subsection{The Monsky trace formula}
	\label{subsection: MW trace formula}
	We now give a brief overview of the Monsky trace formula
	for curves. For a complete treatment, see \cite{Monsky-forma_cohomology3} 
	or \cite[\S 10]{Wan-higher_rank_dwork_conjecture}.  
	Let
	$\Omega_{\mathcal{B}^\dagger}^i$ denote the space of
	$i$-forms of $\mathcal{B}^\dagger$ (see \cite[\S 4]{Monsky-forma_cohomology3}).
	The map $\sigma$ induces a map $\sigma_i: \Omega^i_{\mathcal{B}^\dagger} \to 
	\Omega^i_{\mathcal{B}^\dagger}$
	that sends $xdy$ to $x^\sigma d(y^\sigma)$. Following
	\cite[\S 3]{van_der_Put-MW_cohomology}, there are trace maps
	$
	\text{Tr}_i: \Omega_{\mathcal{B}^\dagger}^i \to 
	\sigma(\Omega_{\mathcal{B}^\dagger}^i)$.
	We then let $\Theta_i$ denote the map $\sigma_i^{-1} \circ \text{Tr}_i$.
	Note that for $\omega \in \Omega^1_{\mathcal{B}^\dagger}$ and $x \in \mathcal{B}^\dagger$ we have
	\begin{align} \label{Dwork operator property on forms}
	\Theta_1(x\omega^\sigma) &= \Theta_0(x)\omega.
	\end{align}	
	\noindent Now consider the $L$-function
	\begin{align} \label{introduction of L-function:2}
	L(\rho,V,s)&= \prod_{x \in V} \frac{1}{1 - \rho(Frob_x) s^{\deg(x)}},
	\end{align}
	which is a slight modification of \eqref{introduction of L-function}. 
	The Monsky trace formula states
	\begin{align} \label{Main Monsky-Washnitzer}
	L(\rho,V,s) &= \frac{P(\Theta_1 \circ \alpha,s)}
	{P(\Theta_0 \circ \alpha,s)}.
	\end{align}
	Thus, we may estimate $L(\rho,V,s)$ by
	estimating operators on the space of $1$-forms and $0$-forms.
	
	In our situation we may simplify \eqref{Main Monsky-Washnitzer}. 
	The ring homomorphism $\mathcal{A}^\dagger \to \mathcal{B}^\dagger$ is \`etale,
	which means $\Omega_{\mathcal{B}^\dagger} = \pi^* \Omega_{\mathcal{A}^\dagger}$. 
	Since
	$\Omega_{\mathcal{A}^\dagger}$ is equal to $\mathcal{A}^\dagger \frac{dT}{T}$, it 
	follows that 
	$\Omega_{\mathcal{B}^\dagger} = \mathcal{B}^\dagger \frac{dT}{T}$.
	Furthermore, as $\frac{dT}{T}= \frac{1}{q}(\frac{dT}{T})^\sigma$ we know by
	\eqref{Dwork operator property on forms} that
	\begin{align*}
	\Theta_1\Big (x\frac{dT}{T}\Big ) &=\frac{1}{q} \Theta_0(x) \frac{dT}{T},
	\end{align*}
	which means $\Theta_1=U_q$ and $\Theta_0=qU_q$. Therefore \eqref{Main 
		Monsky-Washnitzer}
	becomes
	\begin{align} \label{MW via characteristic}
	L(\rho,V,s) &= \frac{P(U_q \circ \alpha, s)}
	{P(U_q\circ \alpha, qs)}.
	\end{align}
	As $P(U_q \circ \alpha, s) \in 
	1+s\mathcal{O}_L[[s]]$, we know $\frac{1}{P(U_q \circ \alpha, qs)}$ lies in $ 1+qs\mathcal{O}_L[[qs]]$.
	This means each slope of $NP_q\Big(\frac{1}{P(U_q \circ \alpha, qs)}\Big)$ is at least one. In particular, we have
	\begin{align} \label{L-function via U_p operator:bad}
	NP_q(L(\rho,V,s))_{<1} &= NP_q(U_q \circ \alpha)_{<1}.
	\end{align}
	By \eqref{introduction of L-function:2}, we know $L(\rho,V,s)$
	has coefficients in $\mathcal{O}_E$. From \eqref{MW via characteristic} we have $P(U_q\circ \alpha,s)=\prod_{i=0}^\infty L(\rho,V,q^is)$,
	so $P(U_q \circ \alpha,s)$ has coefficients in $\mathcal{O}_E$. Then Lemma \ref{lemma: L-function via U_p operator lemma}
	gives
	\begin{align} \label{equation: L function estimate comes from U_p}
	\frac{1}{a}NP_p(U_p \circ \alpha_0)_{<1} &= NP_q(L(\rho,V,s))_{<1}.
	\end{align}

	\subsection{Estimating $NP_p(U_p\circ \alpha_0)$}
	\label{subsection: estimating NP_p}
	We now estimate the $p$-adic Newton polygon of $U_p\circ \alpha_0$
	acting on $\mathcal{B}^\dagger$. To state our main result, we
	recall that $r_*$ is the cardinality of $\eta^{-1}(*)$ and that $d_i$
	is the Swan conductor of $\rho$ at $\tau_i$.

	\begin{proposition}  \label{proposition: NP bounds with operator assumption1}
		The operator $U_p\circ \alpha_0$ is nuclear on $\mathcal{B}^\dagger$ and
		\begin{align*}
		(\frac{1}{a}NP_p(U_p\circ \alpha_0))_{<1} &\succeq \{\underbrace{0,\dots,0}_{g-1+r_0+r_1+r_\infty} \}
		\bigsqcup \Bigg (\bigsqcup_{i=1}^{\mathbf{m}} \Bigg\{ \frac{1}{d_i}, \dots, \frac{d_i-1}{d_i} \Bigg\} \Bigg ).
		\end{align*}
	\end{proposition}
	\noindent The proof is broken into several steps. The main idea is to embed $\mathcal{B}^\dagger$
	into $\mathcal{R}^\dagger$, \emph{twist} the space in some sense,
	and then estimate the columns of an operator on a smaller subspace
	using the computations in \S \ref{subsubsection: the local frob and Up}.
	
	\paragraph{The twisted space}
	We view $\mathcal{B}^\dagger$ as a subspace of $\mathcal{R}^\dagger$ 
	by expanding each function $f \in \mathcal{B}^\dagger$ in terms of each parameter $u_Q$.
	However, the global $p$-Frobenius structure $\alpha_0$ may not have the nice growth
	properties that the local $p$-Frobenius structures we computed in \S \ref{subsection: local rank one crystals} have. The workaround is to ``twist'' the space and the operator. 
	From \eqref{local change of Frobenius eq}, we know that the operator $U_p \circ \vec{c}$
	acts on the space $\vec{b}(\mathcal{B}^\dagger)$ and we have:
	\begin{align}
		P(U_p\circ \alpha_0,s)&=P(U_p \circ \vec{c},s). \label{equation: twisting space}
	\end{align}

	\begin{proposition} \label{local to global: kernel and cokernel}
		Let $\widehat{W}$ (resp. $W^\dagger$) be the space $\widehat{\mathcal{B}}$ or $\vec{b}(\widehat{\mathcal{B}})$ (resp. $\mathcal{B}^\dagger$ or $\vec{b}(\mathcal{B}^\dagger)$). The following hold:
		\begin{enumerate}
			\item We have $pr(W_0^\dagger)= \mathcal{O}_{\mathcal{S}^\dagger}$ 
			and $pr(\widehat{W})= \mathcal{O}_{\mathcal{S}}$.
			\item Both $\ker(pr:W^\dagger \to \mathcal{S}^\dagger)$ and
			$\ker(pr:\widehat{W} \to \mathcal{S})$ have dimension $g-1+r_0+r_1+r_\infty$ as vector
			spaces over $L$.
		\end{enumerate}
	\end{proposition}

	\noindent To prove Proposition \ref{local to global: kernel and cokernel}
	we need the following lemma:
	
	\begin{lemma} \label{computing kernel dimensions by reducing mod pi}
		Let $f: A \to V$ be a continuous map of Banach spaces
		such that $f(A_0) \subset V_0$. If $\overline{f}: \overline{A} \to \overline{V}$ is surjective, 
			then $f$ is surjective and  $f(A_0)=V_0$. Furthermore,
			\begin{align}
			\overline{\ker(f)} &= \ker(\overline{f}). \label{mod p kernel result}
			\end{align}
	\end{lemma}
	\begin{proof}
		Recall that $\pi_\circ$ is a uniformizing element of $\mathcal{O}_L$. Assume 
		$\overline{f}$ is surjective. To prove $f(A_0)=V_0$, take $x \in V_0$, 
		and successively
		approximate $x$ by elements of $f(A_0)$.
		This gives the surjectivity of $f$ as well. To prove \eqref{mod p kernel 
			result}, 
		we 
		write $\overline{A}=\ker(\overline{f})\oplus \overline M$.
		Consider the map $g:\ker(f)_0 \to \overline{A}$.  The image of $g$
		lies in $\ker(\overline{f})$ and the kernel of $g$ is $\pi_\circ\ker(f)_0$.
		This gives an injective map $\overline{\ker(f)} \to
		\ker(\overline{f})$. To show surjectivity, let $\overline{x} 
		\in\ker(\overline{f})$ and let $x$ be a lift
		of $\overline{x}$ to $A_0$. Then $f(x) \in \pi_\circ V_0$.
		Write $f(x)=\pi_\circ y$, where $y \in V_0$.
		There exists $a \in A_0$ such that $f(a)=y$.
		We have $f(x-\pi_\circ a)=0$.
		Thus, $x-\pi_\circ a $ is a lift of $x$ contained in 
		$\ker(f)_0$. 
	\end{proof}

	\begin{proof} (of Proposition \ref{local to global: kernel and cokernel})
		Let us first consider $\widehat{W}$.
		By Lemma \ref{computing kernel dimensions by reducing mod pi} we may prove 
		the 
		corresponding results for $\overline{pr}$.
		We know from \eqref{local change of frob is eq 1 mod p} that 
		$\vec{b} $
		is the identity modulo $\gothm$. In particular, the kernel of the
		projection of $\widehat{\mathcal{B}}$ and $\vec{b}\widehat{\mathcal{B}}$ will
		have the same dimension. We calcuate $\overline{pr}$ on $\overline{B}$
		as follows: expand $g \in \overline{B}$ 
		in terms of the parameter $\overline{u}_Q$ and then truncate the
		Laurent series depending on $\eta(Q)$.
		Let $D$ be the divisor $\sum_{i=1}^{r_1} (p-1)[P_{1,i}]$. 
		The kernel of $\overline{pr}$ is $H^0(X,\mathcal{O}_X(D))$,
		which has dimension $g-1+r_0+r_1+r_\infty$ by Riemann-Roch and 
		\eqref{riemann-hurwitz eq}. Surjectivity also 
		follows from Riemann-Roch. 
		To prove the results for $W^\dagger$, first note
		that $\ker(pr:\mathcal{R} \to \mathcal{S}) \subset \mathcal{R}^\dagger$,
		since this kernel consists of functions with finite order poles. Thus $\ker(pr:W^\dagger \to \mathcal{S}^\dagger)= \ker(pr:\widehat{W} \to \mathcal{S})$.
		Surjectivity follows from the fact that \eqref{overconvergent expansion diagram} is 
		Cartesian. 
	\end{proof}

	\paragraph{Choosing a basis and subspace}
	For the remainder of this section, we let $V=\vec{b}(\mathcal{B}^\dagger)$
	and we let $v=U_p \circ \vec{c}$. From Proposition \ref{local to global: kernel and cokernel}
	we know that condition \eqref{condition: surjective on tails with finite coker} is satisfied. We let $I$ and $B(I)$ be as in \S \ref{subsubsection: Dwork operators are nuclear}. We will estimate
	$P(v,s)$ by estimating a Fredholm determinant on a subspace of $V$. 
	The first step is to describe
	$x=(x_i)_{i \in I} \in \mathbf{c}(I)$ such that $\mathbf{s}(I,x)_0 =V\cap \mathcal{W}$. 
	We break up the definition of $x_i$ into four cases: the
	first case is when $i \in K$ and each the other three cases correspond to
	the three types of points $Q$ described in \S \ref{subsubsection: the local frob and Up}. 
	\begin{align} \label{X definition}
		x_i &= \begin{cases}
			1 & i \in K \\
			\pi^{pn} & i=(n,Q),~ \nu(Q) =0,\infty  \text{ and $\rho_Q$ is unramified} \\
			\pi_{d_l}^{pn} & i=(n,Q),~ \nu(Q) =0,\infty  \text{ and $\rho_Q$ is ramified with $Q=\tau_l$} \\
			p^{a(n)} & i=(n,Q) \text{ and } \nu(Q)=1.
		\end{cases}
	\end{align}
	From the definition of $\mathcal{W}$ we see that
	$s(I,x)_0 =V\cap \mathcal{W}$.
	Indeed, we just selected the $x_i$ appropriately for each summand in the definition
	of $\mathcal{W}$.

	\paragraph{Estimating the column vectors}
	We now estimate $\mathbf{col}_{(i,1)}(v,B(I_E,x))$ for each $i \in I$. 
	We break this up into the four cases used when defining $x_i$.
	\begin{enumerate}[label=\Roman*.]
		\item For $i \in K$, we have $x_ie_i=e_i$. We know from \eqref{eq how the basis looks} that
		$e_i \in \mathcal{W}$. By \eqref{equation: W is preserved} we know $v(\mathcal{W}) \subset \mathcal{W}$, which means
		$v (e_i) \in s(I,x)_0$. Thus, $\mathbf{col}_{(i,1)}(v,B(I_E,x))\geq 0$
		and 
		\begin{align*}
			\{\mathbf{col}_{(i,1)}(v,B(I_E,x))\}_{i \in K} \succeq \{\underbrace{0,\dots,0}_{g-1+r_0+r_1+r_\infty} \}.
		\end{align*}
		\item Fix $Q$ with $\eta(Q)=1$ and consider $i=(n,Q)\in J$. 
		Recall that $x_i=p^{a(n)}$ and $e_i=u_Q^{-n} + c_i$. 
		Write $n=j+pm$, where $0 \leq j <p$. By \eqref{equation: W is preserved} we have $v(x_ic_i) \in p^{a(n)}\mathcal{W}$ 
		and by \eqref{UP computation type 1} we have $v(x_iu_Q^{-n}) \subset p^m\mathcal{W}$.
		From the definition of $a(n)$, we see that $a(n)\geq m$, which implies
		$v (x_ie_i) \in p^m\mathcal{W}$. Therefore $\mathbf{col}_{(i,1)}(v,B(I_E,x))\geq m$.
		We only consider the pairs $(Q,n)$ with $n\geq p$ by \eqref{J definition}. This gives:
		\begin{align*}
			\{\mathbf{col}_{((n,Q),1)}(v,B(I_E,x))\}_{n\geq p} \succeq  \{1,2,3,\dots \}^{ \times p}.
		\end{align*}
		\item Fix $Q\in W$ such that $\eta(Q)=0,\infty$ and $\rho_Q$
		is unramified. Consider $i=(n,Q)\in J$. By \eqref{J definition} we only need to consider pairs $(n,Q)$ where
		$n\geq 1$. 
		From \eqref{X definition}, \eqref{UP computation type 2} and \eqref{equation: W is preserved}, we see that $v(x_ie_i) \in p^n\mathcal{W}$. This gives:
		\begin{align*}
		\{\mathbf{col}_{((n,Q),1)}(v,B(I_E,x))\}_{n\geq 1} \succeq  \{1,2,3,\dots \}.
		\end{align*}
		\item Finally,  fix $Q\in W$ such that $\eta(Q)=0,\infty$ and $\rho_Q$
		is ramified. Then $Q=\tau_l$ for some $1\leq l \leq {\mathbf{m}}$. Consider $i=(n,Q)\in J$. Again, we only consider pairs $(n,Q)$ where
		$n\geq 1$. From \eqref{X definition}, \eqref{UP computation type 3} and \eqref{equation: W is preserved}, we have $v(x_ie_i) \in \pi_{d_l}^{(p-1)n}\mathcal{W}$. As $v_p(\pi_{d_l}^{(p-1)n})=\frac{n}{d_l}$  we obtain:
		\begin{align*}
		\{\mathbf{col}_{((n,Q),1)}(v,B(I_E,x))\}_{n\geq 1} \succeq  \{\frac{1}{d_l},\frac{2}{d_l}, \dots \}.
		\end{align*}
	\end{enumerate}
	\noindent 
	Recall that there are $\mathbf{m}$ points where $\rho_Q$ is ramified.
	We put everything together to get:
	\begin{align*}
		\{\mathbf{col}_{(i,1)}(v,B(I_E,x))\}_{i \in I} \succeq \{\underbrace{0,\dots,0}_{g-1+r_0+r_1+r_\infty} \}
		\bigsqcup  \Bigg\{1,2,\dots \Bigg\}^{\times(pr_1 + r_0 + r_\infty - {\mathbf{m}})} 
		\bigsqcup \Bigg (\bigsqcup_{i=1}^{\mathbf{m}} \Bigg\{ \frac{1}{d_i}, \frac{2}{d_i}, \dots \Bigg\} \Bigg ).
	\end{align*}
	From Remark \ref{remark:tightness on first coordinate} we see that $v$ is
	tight. Since $v$ is a $p$-Dwork operator and \eqref{condition: surjective on tails with finite coker} is satisfied, 
	Proposition \ref{proposition: NP bounds with operator assumption1}
	follows from Lemma \ref{lemma: estimating NP by estimating columns} and Proposition \ref{corollary: completely continuous  on small radius}.

		\subsection{Finishing the proof}
		
		We now finish the proof of Theorem \ref{main theorem}. 
		From \eqref{equation: L function estimate comes from U_p}
		and Proposition \ref{proposition: NP bounds with operator assumption1}
		we know
		\begin{align} \label{lower bound for NP 1}
		NP_q(L(\rho,V,s))_{<1} & \succeq  \big 
		\{\underbrace{0,\dots,0}_{g-1+r_0+r_1+r_\infty}
		\big \}
		\bigsqcup \Bigg ( \bigsqcup_{i=1}^{\mathbf{m}}\Bigg\{ \frac{1}{d_{i}}, \dots, 
		\frac{d_i-1}{d_i}\Bigg\} \Bigg ).
		\end{align}
		Comparing \eqref{introduction of L-function} with \eqref{introduction of L-function:2}
		gives
		\begin{align*} 
		L(\rho,V,s)&= L(\rho,s)\cdot \prod_{\stackrel{R \in W }{R \neq \tau_i}} 
		(1-\rho(Frob_R)s).
		\end{align*}
		This product has $r_0+r_1+r_\infty-{\mathbf{m}}$ terms. Each term accounts for a 
		slope zero segment. Thus,
		\begin{align*} 
		NP_q(L(\rho,s))_{<1} \succeq  \big 
		\{\underbrace{0,\dots,0}_{g-1+{\mathbf{m}}}
		\big \}
		\bigsqcup \Bigg ( \bigsqcup_{i=1}^{\mathbf{m}} \Bigg\{ \frac{1}{d_{i}}, \dots, 
		\frac{d_i-1}{d_i}\Bigg\} \Bigg ).
		\end{align*}	
		We know that the degree of $L(\rho,s)$ is 
		$2(g-1+{\mathbf{m}}) + \sum (d_{i} - 1)$, which accounts for the remaining slope one 
		segments. This completes the proof.

	\bibliographystyle{plain}
	\bibliography{bibliography.bib}

\end{document}